\documentclass[11pt]{article}

\usepackage[english]{babel}
\usepackage[T1]{fontenc}
\usepackage[utf8]{inputenc}
\usepackage{amsbsy}
\usepackage{amsmath}
\usepackage{amssymb}
\usepackage{amsfonts}
\usepackage{amsthm}
\usepackage{bm}
\usepackage{graphicx}
\usepackage{hyperref}
\usepackage[active]{srcltx}
\usepackage{upgreek}
\usepackage{xcolor}
\usepackage[all]{xy}
\usepackage{dsfont}

\usepackage{comment}

\usepackage{tikz}

\newcommand{\blue}[1]{\color{blue}#1\color{black}}

\newcommand{\red}[1]{\color{red}#1\color{black}}

\newcommand{\N}{\mathbb{N}}

\newcommand{\R}{\mathbb{R}}

\newcommand{\Z}{\mathbb{Z}}

\newcommand{\gA}{\mathfrak{A}}

\newcommand{\ga}{\mathfrak{a}}
\newcommand{\gb}{\mathfrak{b}}
\newcommand{\gB}{\mathfrak{B}}
\newcommand{\gc}{\mathfrak{c}}
\newcommand{\gd}{\mathfrak{d}}

\newcommand{\gu}{\mathfrak{u}}

\newcommand{\grad}{\nabla}

\newcommand{\energyset}{\mathcal{X}(\R)}
\newcommand{\energysethydro}{\mathcal{X}_{hy}(\R)}

\newcommand{\energysethydrok}{\mathcal{X}_{hy}^l(\R)}
\newcommand{\Nenergyset}{\mathcal{NX}(\R)}
\newcommand{\Nenergysethydro}{\mathcal{NX}_{hy}(\R)}
\newcommand{\Nenergysethydrok}{\mathcal{NX}_{hy}^l(\R)}

\newcommand{\Ror}{R^{or}}
\newcommand{\Rhy}{R^{hy}}

\newcommand{\grandOde}[1]{\mathcal{O}\left( #1\right)}

\newcommand{\ntend}{\underset{n\rightarrow +\ii}{\longrightarrow}}

\newcommand{\ttendii}{\underset{t\rightarrow +\ii}{\longrightarrow}}

\newcommand{\tiitendf}{\underset{t\rightarrow +\ii}{\rightharpoonup}}

\newcommand{\tiitendfd}{\overset{d}{\underset{t\rightarrow +\ii}{\rightharpoonup}}}

\newcommand{\ntendfX}{\overset{\mathcal{X}_{hy}}{\underset{n\rightarrow +\ii}{\rightharpoonup}}}

\newcommand{\tiitendfX}{\overset{\mathcal{X}_{hy}}{\underset{t\rightarrow +\ii}{\rightharpoonup}}}

\newcommand{\ii}{\infty}

\newcommand{\adm}{\mathrm{Adm}}
\newcommand{\pos}{\mathrm{Pos}}

\newcommand{\loc}{\mathrm{loc}}

\newcommand{\norm}[2]{\left\Vert #1\right\Vert_{#2}}
\newcommand{\normX}[1]{\left\Vert #1\right\Vert_{\mathcal{X}_{hy}}}
\newcommand{\normLii}[1]{\left\Vert #1\right\Vert_{L^\ii}}
\newcommand{\normLdeux}[1]{\left\Vert #1\right\Vert_{L^2}}

\newcommand{\normHun}[1]{\left\Vert #1\right\Vert_{H^1}}
\newcommand{\psLdeux}[2]{\left\langle #1,#2 \right\rangle_{L^2}}
\newcommand{\psLdeuxLdeux}[2]{\left\langle #1,#2 \right\rangle_{L^2\times L^2}}
\newcommand{\vvvert}[1]{\big\vert\kern-0.25ex\big\vert\kern-0.25ex\big\vert #1\big\vert\kern-0.25ex\big\vert\kern-0.25ex\big\vert}

\newcommand{\normLdeuxLdeux}[1]{\left\Vert #1\right\Vert_{L^2\times L^2}}

\newcommand{\normR}[1]{\left| #1\right|}

\newcommand{\QN}{\mathcal{Q}^{(N)}}

\newtheorem{thm}{Theorem}[section]

\newtheorem{claim}[thm]{Claim}

\newtheorem{lem}[thm]{Lemma}
\newtheorem{prop}[thm]{Proposition}

\newtheorem*{thm*}{Theorem}
\newtheorem*{merci}{Acknowledgments}
\newtheorem{rem}[thm]{Remark}

\theoremstyle{remark}

\title{Construction of a multi-soliton-like solutions for non-integrable Schrödinger equations with non-trivial far field}

\begin{document}

\date{}
\author{
\renewcommand{\thefootnote}{\arabic{footnote}}
Jordan Berthoumieu\footnotemark[1]}
\footnotetext[1]{Institut de Recherche Mathématique Avancée, UMR 7501 Université de
Strasbourg, 7 rue René Descartes 67000 Strasbourg, France. Website: https://berthoumieujordan.wordpress.com \ E-mail: {\tt berthoumieu@unistra.fr}}

\maketitle

\begin{abstract}

This article provides a naturel sequel of previous works~\cite{Bert2,Bert3} regarding the stability of travelling waves for a general one-dimensional Schrödinger equation~\eqref{NLS} with non-zero condition at infinity. The aim of this article is twofold. First, we prove the asymptotic stability of well-prepared chains of dark solitons and secondly, we construct an asymptotic $N$-soliton-like solution, which is an exact solution of~\eqref{NLS}, the large-time dynamics of which is similar to a decoupled chain of solitons.
\end{abstract}

\numberwithin{equation}{section}

\section{Introduction}\label{section: intro}

We are interested in the defocusing nonlinear Schrödinger equation 
\begin{equation}\label{NLS}\tag{$NLS$}
    i\partial_t \Psi +\partial_x^2 \Psi +\Psi f(|\Psi|^2)=0\quad\text{on }\R\times\R,
\end{equation}
with the condition at infinity
\begin{equation}\label{nonvanishing condition à l'infini}\tag{$NVC$}
    |\Psi(t,x)|\underset{|x|\rightarrow +\ii}{\longrightarrow}1.
\end{equation}

The system~\eqref{NLS} endowed with~\eqref{nonvanishing condition à l'infini} should be viewed as a generalization of the Gross-Pitaevskii equation corresponding to the nonlinearity $f(\rho)=1-\rho$. Models such as~\eqref{NLS} with non-trivial far field are relevant in condensed matter physics and arise, in particular, in the context of the Bose-Einstein condensation or of superfluidity (see~\cite{AbiHuMeNoPhTu,Coste1}). They also appear in nonlinear optics to describe the propagation of waves in nonlinear media (see~\cite{KivsLut1,KivPeSt1}).

To keep in compliance with condition~\eqref{nonvanishing condition à l'infini}, we shall assume that the nonlinearity $f$ satisfies the condition $f(1)=0$ and we restrict our attention to the defocusing case, namely
\begin{equation*}\label{f'(1) <0}
    f'(1)<0.
\end{equation*}

The choice of general classes of nonlinearities provides original behaviors as enumerated by D. Chiron in~\cite{Chiron7}, encompassing relevant physics models for modeling repulsive interactions between bosons in Bose-Einstein condensates or for the propagation of waves in anomalous media in the context of optical fibers. In the latter framework, the non-trivial far field~\eqref{nonvanishing condition à l'infini} expresses the presence of a nonzero background. The dynamics of solutions differs drastically from the case of null condition at infinity, in the sense that the resulting dispersion relation is different and that we expect the general dynamics to be governed by scattering for null condition at infinity (see~\cite{GiniVel02} for instance). In contrast, special localized and coherent structures emerge naturally in this nonlinear dispersive equation. They are called travelling waves (or solitons). We also find in the literature the terminology of \textit{dark or gray solitons} since these solutions appear as a local decline of the light intensity in a bright optical fiber.

Generically, dispersive systems seem to be governed by the \textit{soliton resolution conjecture}, according to which a solution with localized and smooth initial data eventually decomposes as a chain of decoupled solitons and an extra dispersive tail, that travels even faster than the soliton in our particular case\footnote{This is a consequence of the dispersion relation~\eqref{dispersion relation}.}. This decomposition was highlighted for several integrable dispersive models, by the inverse scattering method in~\cite{EckhSch1,jenliupersul}, but also by more sophisticated methods, for instance in~\cite{CuccJen1,BorJenMcL} and very recently in~\cite{GassGeraMill1}.

In the sequel, we somehow finish the preliminaries of the study of the \textit{stability} aspects made in the previous articles~\cite{Bert1,Bert2,Bert3}. In fact, we show the asymptotic stability of a well-prepared chain of travelling waves, that is a decoupled sum of travelling waves. Then, we prove the existence of an asymptotic $N$-soliton, which is an exact solution of~\eqref{NLS} that behaves as a decoupled sum of travelling waves of different speeds, also called \textit{chain of well-prepared solitons} (see~\eqref{definition N-solition} below).

Regarding the nonlinearity $f$, it is taken as general as possible. Among all the benefits provided by studying the mechanism of stability for such a wide class of generality, it is of great interest to extend the asymptotic description of solutions to non-integrable systems. In fact, it is expected that non-integrable systems share numerous similarities with integrable systems, especially when it comes to the large-time dynamics of its underling waves and the inclination to decompose according to the soliton resolution conjecture. Only a few results are known for non-integrable systems, but we can cite very recent results such as~\cite{duykemer1,JendLaw2} for the nonlinear wave equation and~\cite{duykemermar,JendLaw1} for the critical wave map equation. In this sense, the assumptions on the nonlinearity are fairly large, in order to cover (in principle) non-integrable systems. In the sequel, we assume that the nonlinearity $f\in C^3(\R_+)$ satifies
\begin{itemize}
    \item For all $\rho\in\R$,
\begin{equation}\label{hypothèse de croissance sur F minorant intermediaire}\tag{H1}
    \dfrac{c_s^2}{4} (1-\rho)^2 \leq F(\rho).
\end{equation}
    \item There exist $M\geq 0$ and $q\in [2, +\ii)$ such that for all $\rho \geq 2$,
\begin{equation}\label{hypothèse de croissance sur F majorant}\tag{H2}
    F(\rho)\leq M|1-\rho|^q.
\end{equation}
    \item \begin{equation}\label{condition suffisante pour la stabilité orbitale sur f''(1)+6f'(1)>0}\tag{H3}
    f''(1)+3f'(1)\neq 0.
\end{equation}
\end{itemize}

We lay emphasis on the wide variety of physically relevant nonlinearities that fall within these assumptions. An interesting class of functions are perturbations of the Gross-Pitaevskii case :
\begin{equation*}
f(\rho)=1-\rho + a(1-\rho)^{2p-1}\quad\text{where }p\in [2,+\ii[\text{ and }0<a< p\left(\dfrac{2p-1}{2p-3}\right)^{2p-3},
\end{equation*}

as well as the pure-power type nonlinearities
\begin{equation*}
f(\rho)=\alpha(1-\rho^\beta)\quad\text{where }\alpha>0\text{ and }\beta >1,
\end{equation*}

or even the saturated nonlinearities
\begin{equation*}
f(\rho)=\dfrac{1}{(1+\gamma\rho)^2}-\dfrac{1}{(1+\gamma)^2}\quad\text{where }\gamma >0.
\end{equation*}

We refer to~\cite{AntHieMar} (see Exemple 1.9) and the references therein for more details on the relevance and the origin of these special nonlinearities from the physics literature. Furthermore, the previous classes appear to be well adapted in the context of Theorem~\ref{théorème: il existe une branche C^1 de solitons proche de c_s}. In these special classes, a more explicit description of the range of existing unique travelling waves can be achieved. We refer to~\cite{Bert2} (see Subsection 1.3) for more explicit comments on these aspects.

\subsection{Dynamical setting}\label{section: dynamical setting}

The Hamiltonian structure of~\eqref{NLS} is given by the generalized Ginzburg-Landau energy,
\begin{equation*}\label{définition de l'energie de Ginzburg Landau}
    E(\Psi):=E_k(\Psi)+E_p(\Psi):=\dfrac{1}{2}\int_\R |\partial_x \Psi|^2+ \dfrac{1}{2}\int_\R F(|\Psi|^2)\quad\text{with }F(\rho):=\int_\rho^1 f(r)dr.
\end{equation*}

The energy is well-defined on the natural energy set 
\begin{equation*}\label{energy space, espace d'energie}
    \mathcal{X}^0(\R)=\energyset:=\big\lbrace \psi\in H^1_{\loc}(\R) \big| \psi' \in L^2 (\R), F(|\psi|^2) \in L^1(\R) \big\rbrace ,
\end{equation*}

and on 
\begin{equation*}\label{energy space, espace d'energie non nul}
    \Nenergyset:=\energyset\cap \big\{ \inf_\R |\psi| >0\big\},
\end{equation*}

endowed with the distance
\begin{equation}\label{métrique d}
    d(\psi_1,\psi_2)=\left\Vert \psi_1-\psi_2\right\Vert_{L^\ii([-1,1])}+\normLdeux{|\psi_1|^2 - |\psi_2|^2}+\normLdeux{\psi_1'-\psi_2'}.
\end{equation}

We mention that the Cauchy problem for~\eqref{NLS} is wellposed in $u_0 + H^1(\R)$ with initial data $u_0$ in $\mathcal{X}(\R)$ and has been handled under several conditions on $f$ that are described below. 
\begin{thm}[\cite{Gallo1}]\label{théorème: local global well-posedness of cauchy problem}
Let $u_0\in \mathcal{X}(\R)$. Take $f$ in $ C^2(\R)$ satisfying \eqref{hypothèse de croissance sur F minorant intermediaire}. In addition, assume that there exist $\alpha_1\geq 1$ and $C_0 >0$ such that for all $\rho\geq 1$, 
\begin{equation}\label{theoreme de gallo condition de croissance sur f''}
|f''(\rho)|\leq \dfrac{C_0}{\rho^{3-\alpha_1}}.\tag{H0}
\end{equation}

If $\alpha_1 > \frac{3}{2}$, assume moreover that there exists $\alpha_2\in [ \alpha_1-\frac{1}{2},\alpha_1]$ such that for $\rho\geq 2$, $C_0 \rho^{\alpha_2} \leq F(\rho)$.\\
There exists a unique function $w \in  C^0\big(\R,H^1(\R)\big)$ such that $u:=u_0+w$ solves \eqref{NLS}, the solution continuously depends on the initial condition, and the energy $E$ is conserved by the flow.
\end{thm}

Observing moreover that $\mathcal{X}(\R)+H^1(\R)\subset \mathcal{X}(\R)$, we infer that the global well-posedness holds in $\mathcal{X}(\R)$. Note that this property was already highlighted to deal with the Gross-Pitaevskii case in~\cite{Gerard1}.

\subsection{Travelling wave solutions and chain of well-prepared solitons}\label{section: minimizing property of the travelling waves}

The linearization of~\eqref{NLS} around the trivial solution $\Psi\equiv 1$ reduces, in the long wave approximation, to the free wave equation. More precisely, the plane wave solutions of this linearized system satisfy the dispersion relation 
\begin{equation}\label{dispersion relation}
    \omega(\xi)=\pm\sqrt{\xi^4+c_s^2 \xi^2}
\end{equation} with sound speed 
\begin{equation*}\label{relation entre vitesse c_s et f'(1)}
c_s = \sqrt{-2 f'(1)}.
\end{equation*}

This quantity $c_s$ plays a crucial role in the understanding of the travelling waves solution. We focus on solutions of the form $\Psi(t,x)=u(x-ct)$. The profile of which satisfies the ordinary differential equation
\begin{equation}\label{TWC}\tag{$TW_c$}
    -icu'+u''+uf(|u|^2)=0.
\end{equation}

We recall that there exists a smooth continuous branch of travelling waves in the transonic limit, i.e. with speed close to $c_s$, which are asymptotically stable in the following sense.
\begin{thm}[\cite{Bert2,Bert3}]\label{théorème: il existe une branche C^1 de solitons proche de c_s}
    There exist a critical speed $c_{\min} >0$ and an integer $l_*$ such that for $c\in (c_{\min},c_s)$ and $f\in C^{l_*}(\R_+)$ satisfying the hypothesis~\eqref{theoreme de gallo condition de croissance sur f''}-\eqref{condition suffisante pour la stabilité orbitale sur f''(1)+6f'(1)>0}, there exists a non-constant and non-vanishing smooth solution $\gu_c$ of \eqref{TWC}, that is unique up to translations and constant phase shifts. Moreover, the travelling wave of speed $c\in (c_{\min},c_s)$ is asymptotically stable in the following sense.

There exists $\delta_{*} >0$ such that for any $\psi_0\in\energyset$ satisfying $d(\psi_0,\gu_{c^*})\leq \delta_{*}$, the solution $\psi(t)$ associated with the initial condition $\psi_0$ is globally wellposed and there exist $c^*_0\in (c_{\min},c_s) $ and two functions $(b,\theta)\in C^1(\R_+,\R^2)$ such that 
    \begin{equation*}
        e^{-i\theta(t)} \psi\big( t, .+b(t)\big)\tiitendfd \gu_{c^*_0},
    \end{equation*}    
    and
    \begin{equation*}
        b'(t)\ttendii c^*_0\quad\quad\text{ and }\quad\quad \theta'(t)\ttendii 0.
    \end{equation*}

Here the notation $\Psi(t)\tiitendfd\Psi_\ii$ stands for the fact that all three following convergences hold \begin{equation}\label{définition convergence faible d}
\left\{\begin{array}{l}
    1-|\Psi(t)|^2\tiitendf 1-|\Psi_\ii|^2\quad\text{in }L^2(\R),\\
    \partial_x \Psi(t)\tiitendf\partial_x \Psi_\ii\quad\text{in }L^2(\R), \\
    \Psi(t)\ttendii\Psi_\ii\quad\text{in }L^\ii_\loc(\R) . \\
\end{array}
\right.
\end{equation}

\end{thm}

\begin{rem}\label{remarque: travelling wave onde progressive en 0}
By uniqueness of a travelling wave, up to phase shifts and translations, we can fix two additional constraints. In particular, for a travelling wave of admissible speed $c\in (c_{\min},c_s)$ that can be lifted as $\gu_c=|\gu_c|e^{i\varphi_c}$, we can renormalize it to simplify the study. We assume that the travelling waves are centered, i.e. that for any $c\in (c_{\min},c_s),\min_\R |\gu_c| = |\gu_c(0)|$ and $\varphi_c(0)=0$.
\end{rem}

In view of the previous result, the next step in the stability analysis of solitons is to consider a chain of well-prepared solitons. We\footnote{Henceforth, the notation $\gc,\ga$ is prescribed for $N$-dimensional vectors and script letters for real numbers.} introduce the set of admissible speeds
\begin{equation*}\label{définition ensemble des vitesses admissibles}
    \adm_N:=\big\{\gc:=(c_1,...,c_N)\in (c_{\min},c_s)^N\big| \ c_1 < ...< c_N\big\}.
\end{equation*}
and the set
\begin{equation*}\label{définition ensemble des positions a_k en L}
    \pos_N(L):=\big\{ \ga:=(a_1,...,a_N)\in\R^N\big| a_{k+1}-a_k >L,\ \forall k\in\{1,...,N-1\}\big\},
\end{equation*}

for $L>0$. For $\gc=(c_1,...,c_N)$, $\ga=(a_1,...,a_N)\in\R^N$ and $\Theta\in\R$, and from the polar decomposition of the corresponding travelling waves $\gu_{c_k,a_k} = |\gu_{c_k,a_k}| e^{i\varphi_{c_k,a_k}}=|\gu_{c_k}(.-a_k)| e^{i\varphi_{c_k}(.-a_k)}$, we define a well-prepared chain of solitons/travelling waves as
\begin{equation}\label{definition N-solition}
    \Ror_{\gc,\ga,\Theta}:=\sqrt{1-\sum_{k=1}^N\big(1-|u_{c_k,a_k}|^2\big)}\ \exp\left(i \displaystyle\int_0^x \sum_{k=1}^N \varphi_{c_k,a_k}\right)\exp(i\Theta)
\end{equation}

or, in a more condensed writing, $\Ror_{\gc,\ga,\Theta}:=\sqrt{1-\eta_{\gc,\ga}}\ e^{-i \int_0^x v_{\gc,\ga}}e^{i\Theta}$,
where the hydrodynamical variables read $(\eta_{\gc,\ga},v_{\gc,\ga})=\sum_{k=1}^N (\eta_{c_k,a_k},v_{c_k,a_k})=\sum_{k=1}^N\big(1-|\gu_{c_k,a_k}|^2, - \varphi'_{c_k,a_k}\big)$ (see Subsection~\ref{section: Hydrodynamical setting}). We expect such a configuration to decouple as time gets large. As a matter of fact, the mutual interaction between the travelling waves shall become all the more negligible, since they are initially ordered and apart from each other (as it is illustrated in Figure~\ref{figure: chaine de trois solitons intro}); thus they will drift apart even more as time evolves. One step was already made in this direction in~\cite{Bert2}, in which the orbital stability of such a configuration was demonstrated.

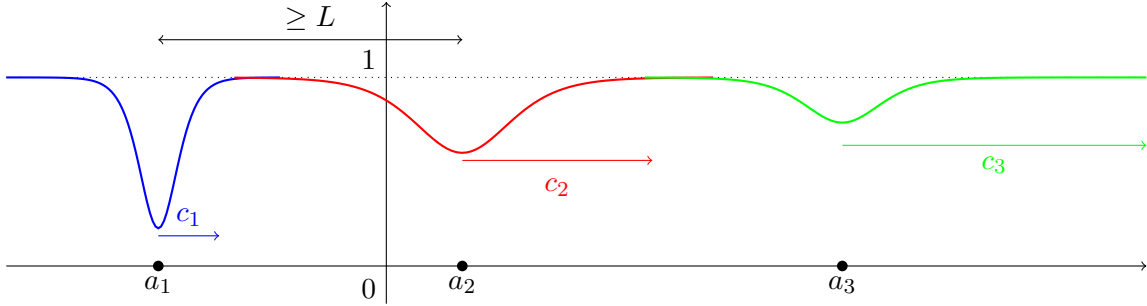
\begin{figure}[h!]
    \centering
    \begin{tikzpicture}

% Axes

\draw[->] (4,0) -- (4,4) node[left] {};
\draw[->] (-1,0.5) -- (14,0.5) node[below] {};

% Ligne y = 1 (pointillée)
\draw[dotted,black] (-1,3) -- (14,3);

% Labels 0 et 1
\node[left] at (4.0,0.2) {$0$};
\node[left] at (4.0,3.23) {$1$};

% Courbes sech^2
\draw[thick,blue,domain=-1:2.6,samples=100] 
  plot (\x,{3-2*(1/cosh(3*(\x-1)))^2}) node[right] {}; % Courbe C_1
\draw[thick,red,domain=2:8.3,samples=100] 
  plot (\x,{3-(-1/cosh(1.2*(\x-5)))^2}) node[right] {};
\draw[thick,green,domain=7.4:14,samples=100] 
  plot (\x,{3-0.6*(1/cosh(1.5*(\x-10)))^2}) node[right] {};

% Points et annotations
\fill (1,0.5) circle (2pt) node[below] {$a_1$};
\fill (5,0.5) circle (2pt) node[below] {$a_2$};
\fill (10,0.5) circle (2pt) node[below] {$a_3$};

%vitesses
\draw[->,blue] (1,0.9) -- (1.8,0.9) node[below] {};
\draw[->,red] (5,1.9) -- (7.5,1.9) node[below] {};
\draw[->,green] (10,2.1) -- (14,2.1) node[below] {};

\node[above] at (1.4,0.9) {\blue{$c_1$}};
\node[above] at (6.25,1.3) {\red{$c_2$}};
\node[above] at (12,1.6) {\textcolor{green}{$c_3$}};

% Distance L
\draw[<->] (1,3.5) -- (5,3.5) node[midway,above] {$\geq L$};

\end{tikzpicture}
    \caption{\textit{The modulus $\sqrt{1-\eta_{\gc,\ga}}$ of a well-prepared chain of three solitons, with speeds $c_1 < c_2 < c_3$, and $(a_1,a_2,a_3)\in\pos_3(L)$.}}\label{figure: chaine de trois solitons intro}
\end{figure}

\subsection{Main results}

A chain of travelling waves needs not be an exact solution because of the nonlinear nature of~\eqref{NLS}. However, it is of interest to construct a pure multi-soliton, i.e. an exact solution of~\eqref{NLS} that decouples asymptotically into such a decoupled sum of solitons as $t\rightarrow\pm\ii$, thus, that is compliant with the structure of the chain of well-prepared solitons studied in the preceding section. The main feature of a pure multi-soliton is that the potential collisions and interactions between the solitons that emerge generally preserve the shape of each soliton, namely its speed/amplitude. Such a configuration can be constructed by the inverse scattering method for integrable systems. In comparison, in this framework, we deal with a general nonlinearity $f$, which yields a (in-principle) non-integrable equation. Thus we rely on a different method, first introduced in~\cite{Martel1} and adapted to the general one-dimensional nonlinear Schrödinger equation with zero condition at infinity in~\cite{MartMer5} and very recently to the one-dimensional Zakharov system in~\cite{Riall1}.

For complete integrable systems, we expect the existence of pure multi-soliton (see~\cite{AbloSeg1} for explicit examples), due to the elastic nature of the collisions between the travelling waves. Furthermore, it is conjectured that the integrable structure is related with the nature of collisions but there are only a few results in this direction, such as~\cite{MartMer8,munoz2,perel1,Martmer12,PilVal1}.

In order to give a more accurate understanding of the relations between integrability and the elasticity of collisions, we continue the analysis of the multi-soliton-like structure in a very general frame, without further assumptions on the nonlinearity $f$. The main result of this article is the construction of an \textit{asymptotic $N$-soliton} and this is the point of the next theorem.
\begin{thm}\label{théorème: existence d'un asymptotic N soliton like solution classique variable}
    Assume that $f\in C^{3}(\R_+)$ satisfies~\eqref{theoreme de gallo condition de croissance sur f''}-\eqref{condition suffisante pour la stabilité orbitale sur f''(1)+6f'(1)>0}. Let $\gc^*\in (c_{\min},c_s)^N$ and $\ga^*\in\R^N$. There exist a solution $\mathcal{U}^{(N)}\in C^0\big(\R,\mathcal{NX}(\R)\big)$ to~\eqref{NLS} and a function $\Theta\in C^1(\R)$ such that
\begin{equation*}
d\left(\mathcal{U}^{(N)}(t),R^{or}_{\gc^*,\ga^*+\gc^*t,\Theta(t)}\right)\ttendii 0.
\end{equation*}
Moreover, the convergence is exponential in time.
\end{thm}

\begin{rem}
    Note that no assumption is made on $\gc^*$ and $\ga^*$, apart from the fact that we take speeds in the admissible range $(c_{\min},c_s)$, so that each travelling wave is individually stable. Up to ordering the speeds, the sum of solitons decouples as time gets large, the fastest travelling wave ahead of the others, followed by the second fastest, and so forth.
\end{rem}

In addition, we complete the analysis of the stability theory for~\eqref{NLS}, by proving the asymptotic stability of a well-prepared chain of solitons. 
\begin{thm}\label{théorème: stab asymp classical of a chain}
    Assume that the same assumptions that in Theorem \eqref{théorème: il existe une branche C^1 de solitons proche de c_s} hold true. Let $\gc^*\in \adm_N$, there exist $\delta_{*},L_* >0$ such that for any $\psi_0\in\energyset$ satisfying 
    \begin{equation*}
        d(\psi_0,R^{or}_{\gc^*,\ga^*,\Theta^*})\leq \delta_{*}\quad\text{with }(\ga^*,\Theta^*)\in\pos_N(L_*)\times\R^N,
    \end{equation*}
     the solution $\psi(t)$ associated with the initial condition $\psi_0$ is globally defined and there exist $\gc^\ii=(c_1^\ii,...,c_N^\ii)\in (c_1,c_s)^N $ and two functions $\gb=(b_1,...,b_N)\in C^1(\R_+,\R^N)$ and $(\theta_1,...,\theta_N)\in C^1(\R_+,\R^N)$ such that for any $j\in\{1,...,N\}$,
\begin{equation}\label{e^-itheta(t) psibig( t, .+b(t)big)tiitendfd gu_c^*_0}
        e^{-i\theta_j(t)} \psi\big( t, .+b_j(t)\big)\tiitendfd \gu_{c^\ii_j},
    \end{equation}

    and for any $j\in\{1,...,N\}$,
    \begin{equation*}
        b_j'(t)\ttendii c^\ii_j\quad\quad\text{ and }\quad\quad \theta_j'(t)\ttendii 0.
    \end{equation*}

Furthermore, using the convention $(c_0^\ii,c_{N+1}^\ii)=(c_{\min},c_s)$, consider translation parameters $\gB=(B_1,...,B_N,B_{N+1})\in C^0(\R_+,\R^{N+1})$ satisfying for any $j\in\{1,...,N\}$ and any $t\in\R$,
\begin{equation}\label{B_k(t)<b_k(t)<B_k+1(t),}
    B_j(t)<b_j(t)<B_{j+1}(t),
\end{equation}
and
\begin{equation}\label{c_k-1^ii < liminf_trightarrow +iidfracB_k(t)}
    c_{j-1}^\ii < \liminf_{t\rightarrow +\ii}\dfrac{B_j(t)}{t}\leq \limsup_{t\rightarrow +\ii}\dfrac{B_j(t)}{t}<c_j^\ii.
\end{equation}

Then,  for any $j\in\{1,...,N\}$, there exists a constant $\theta_j^\ii\in\R$ such that
\begin{equation*}\label{e^-itheta(t) psibig( t, .+b(t)big)tiitendfd 1}
         e^{-i\theta_j(t)} \psi\big( t, .+B_j(t)\big)\tiitendfd e^{i\theta_j^\ii}.
    \end{equation*}
\end{thm}

The convergence~\eqref{e^-itheta(t) psibig( t, .+b(t)big)tiitendfd gu_c^*_0} is relative to the metric $d$ given in~\eqref{définition convergence faible d}, and is sharp in the sense that it could not be improved to a strong convergence, without additional regularity or localization assumptions on the initial perturbation. For instance, referring to the last item in definition~\eqref{définition convergence faible d}, the strong convergence in $L^\ii(\R)$ is precluded by the emergence of smaller solitons during the evolution, or by a slow phase winding phenomenon. Regarding both convergences in $L^2$, they could not be improved to strong convergences due to the Hamiltonian nature of the equation. Nevertheless, we think that strong convergences in refined spaces could be achieved. A reasonable conjecture is that there exists well-chosen constants $(\sigma_-,\sigma_+)\in\R^2$ such that the convergences are true when we restrict to strong convergences on the branches $\sigma_-t\leq x\leq\sigma_+ t $.

\begin{rem}
The second part of Theorem~\ref{théorème: stab asymp classical of a chain} relates to the core of the region that separates two travelling waves. It claims that there is no interaction between two well-separated solitons. 
\end{rem}

\begin{rem}
As stated before, this work summarizes the stability results regarding the large-time asymptotics of a solution to a dispersive system that is~\eqref{NLS}. Since some structure has been highlighted, it is suitable for the understanding of more complicated problems, such as the mutual interactions between travelling waves, and more precisely the nature of collisions between two solitons, or even in a $N$-soliton, such as the one constructed in Theorem~\ref{théorème: existence d'un asymptotic N soliton like solution classique variable}.
\end{rem}

This article splits into three sections. In Section~\ref{section: Hydrodynamical setting}, we introduce all the features of the hydrodynamical setting, making possible a crucial change of variables, necessary to prove the main theorems. In Section~\ref{section: Asymptotic stability of a well-prepared chain}, we prove the asymptotic stability of the well-prepared chain, and we finish by the construction of the asymptotic $N$-soliton in Section~\ref{section: Construction of an asymptotic N-soliton}.

\begin{rem}
    Throughout the article, we shall write $g=\grandOde{f}$ or $g\lesssim f$ whenever there exists a constant $K>0$ only depending on $\gc^*$ such that for any $(x,c,n)\in\R\times (0,c_s)\times\N,\left| f(x,c,n)\right|\leq K \left|g(x,c,n)\right|$.
\end{rem}

\section{Hydrodynamical setting}\label{section: Hydrodynamical setting}

If $\Psi$ does not vanish, we can lift it as $\Psi=\rho e^{i\varphi}$. Setting the new variables $\eta=1-\rho^2$ and $v=-\partial_x \varphi$, we obtain the hydrodynamical form of the equation 

\begin{equation}\label{NLShydro}\tag{$NLS_{hy}$}
    \left\{
\begin{array}{l}
    \partial_t\eta =-2\partial_x \big(v(1-\eta)\big), \\
    \partial_t v =-\partial_x \bigg( f(1-\eta)-v^2-\dfrac{\partial_x^2 \eta}{2(1-\eta)}-\dfrac{(\partial_x\eta)^2}{4(1-\eta)^2} \bigg). \\
\end{array}
\right.
\end{equation}

For $l\geq 0$, we shall write $\energysethydrok :=H^{l+1}(\R)\times H^l(\R)$ and endow this space with the associated euclidean norm 
\begin{equation*}
    \Vert (\eta,v)\Vert_{\mathcal{X}_{hy}^l}^2=\Vert \eta\Vert_{H^{l+1}}^2 + \Vert v\Vert_{H^{l}}^2.
\end{equation*}

Let us also introduce the non-vanishing associated subset

\begin{equation*}
    \Nenergysethydrok:=\big\{(\eta,v)\in \energysethydrok\big| \max_\R\eta <1\big\}.
\end{equation*}

We will label $\energysethydro:=H^1(\R)\times L^2(\R)$ and $\Nenergysethydro:=\energysethydro\cap \{\max_\R\eta <1\}$ corresponding to the set $\energysethydrok$ when $l=0$. This functional setting is related to several quantities, which are at least formally, conserved along the flow. These are the energy
\begin{equation}\label{expression de l'énergie en hydro}
    E(Q)=E(\eta,v):=\int_\R e(\eta,v):=\dfrac{1}{8}\int_\R\dfrac{(\partial_x\eta)^2}{1-\eta}+\dfrac{1}{2}\int_\R (1-\eta)v^2+\dfrac{1}{2}\int_\R F(1-\eta).
\end{equation}
and the momentum 
\begin{equation*}\label{expression du moment}
    p(\eta,v):=\dfrac{1}{2}\int_\R \eta v .
\end{equation*}

As mentioned in~\cite{Bert1}, the assumptions that we make on the nonlinearity $f$ give to $\Nenergysethydro$ the status of the energy set for~\eqref{NLShydro}. Note also that, in view of the previous expression, $p$ is smooth on $\energysethydro$. Moreover, if $f\in C^l(\R_+)$, $E$ is $ C^{l+1}$ on $\Nenergysethydro$.

\subsection{Existence of a solution in different frameworks}

Regarding the existence of solutions to~\eqref{NLShydro}, we state an upgraded version of a local well-posedness theorem due to C. Gallo, which we improved by showing the weak continuity of the flow (we refer to Appendix~B in~\cite{Bert3}). 
\begin{thm}[\cite{Gallo3}]\label{théoreme: LWP of NLShydro}
    We assume that $f\in C^3(\R_+)$ is such that for all $\rho\in\R$,
\begin{equation*}
    \dfrac{c_s^2}{4} (1-\rho)^2 \leq F(\rho).
\end{equation*}
Let $l\in\{0,1,2\}$ and let $(\eta_0,v_0)\in \Nenergysethydrok$. There exist $T_{\max} >0$ and a unique solution $(\eta,v)\in C^0\big([0,T_{\max} ),\Nenergysethydrok\big)$ to~\eqref{NLShydro} with initial datum $(\eta_0,v_0)$. The maximal time $T_{\max}$ is continuous with respect to the initial datum and is characterized by
    \begin{equation*}
        \lim_{t\rightarrow T_{\max}^-}\sup_{x\in\R}\eta(t,x) = 1.
    \end{equation*}

The flow map is continuous on $\Nenergysethydrok$, and the energy and the momentum are conserved along the flow.
\end{thm}

The weak continuity of the flow map reads as follows.
\begin{prop}[\cite{Bert3}]\label{proposition: weak continuity of the flow in hydrodynamical setting}
    Let $(Q_{n,0})_n\in\big(\Nenergysethydro\big)^\N$ and $Q_{0}\in\Nenergysethydro$ such that 
    \begin{equation*}
        Q_{n,0}\ntendfX Q_0.
    \end{equation*}
Assume in addition that there exists a constant $M\in |0,1)$ such that the associated solutions $Q_n\in C\big([0,T_{\max,n}),\Nenergysethydro\big)$ and $Q\in C\big([0,T_{\max}),\Nenergysethydro\big)$ to~\eqref{NLShydro} given by Theorem~\ref{théoreme: LWP of NLShydro} satisfy for some $0<T<\liminf_{n\rightarrow+\ii} T_{\max,n}$ and any $(n,t,x)\in \N\times [-T,T]\times \R$,
\begin{equation*}
    \big|\eta_n(t,x)\big|\leq M .
\end{equation*}

Then $0 < T < T_{\max}$ and for any $t\in [-T,T]$,
\begin{equation*}
    Q_n(t)\ntendfX Q(t).
\end{equation*}
\end{prop}

In the proposition just below, we give an explicit correspondence between the original and the hydrodynamical framework. The proof of this next result does not depend on the nonlinearity $f$ and can be recovered exactly as in~\cite{Mohamad1}.

\begin{prop}\label{prop:Mohamad_1_1_f}
For $l\geq 0$, we define the mapping
\[
\Phi : \mathcal{NX}_{hy}^l(\mathbb{R}) \times \mathbb{R}/(2\pi\mathbb{Z})
\longrightarrow \mathcal{X}^l(\R)
\]
with
\[
\Phi\big((\eta,v),\theta\big)(x)
=
\sqrt{1-\eta(x)}
\exp\!\left(
-i\int_0^x v(s)\,ds + i\theta
\right),
\]

Then $\Phi$ is a bijection from
$\mathcal{NX}_{hy}^l(\mathbb{R}) \times \mathbb{R}/(2\pi\mathbb{Z})$
onto $\mathcal{NX}^l(\R)$. Moreover, its inverse is given by
\[
\Phi^{-1}(\psi)
=
\left(
\Big(1-|\psi|^2,\ \dfrac{i\psi'. \psi}{|\psi|^2}\Big), \arg\big(\psi(0)\big)
\right).
\]
\end{prop}

\begin{rem}\label{remarque: non continuité des application Phi}
    An important difficulty due to~\eqref{nonvanishing condition à l'infini} can be highlighted in light of~\cite{Mohamad1}. Note that the mapping $\Phi$ is an homeomorphism when the domain is endowed with the metric $d$ and $\Nenergysethydro\times \R/2\pi\Z$ with the natural norm $\normX{.} + |.|_{\R/2\pi\Z}$. However, $\Phi^{-1}$ is not locally Lipschitz continuous, and counter examples are provided in the same article. These topological pitfalls have been overcome in~\cite{Wegn1}, in the sense that a bilipschitz correspondence has been found, between $\Nenergysethydro$ and a well-chosen metric space, related to the original setting.
\end{rem}

Now we give the correspondence between the solution of~\eqref{NLS} and the solution of~\eqref{NLShydro}. Compared to the Gross-Pitaevskii case, which is considered in~\cite{Mohamad1}, the next proposition depends on $f$ but the only difference lies in the replacement
of the linear term $\eta$ by the general nonlinearity $f(1-\eta)$. However, the regularity assumption $f\in C^{3}(\mathbb{R}_+)$ is enough to repeat the arguments of Section 3.2 in~\cite{Mohamad1} and then give a proper proof of the next result.

\begin{prop}\label{prop:Mohamad_3_6_f}
Let $(\eta,v)\in C^0\big((-T_{\max},T_{\max});\mathcal{NX}_{hy}(\mathbb{R})\big)$ be a solution to~\eqref{NLShydro}. Then there exists a function
\[
\Theta \in C^1\big((-T_{\max},T_{\max}),\R\big)
\]
such that the function $\Psi$ defined by
\[
\Psi(t,x)
=
\sqrt{1-\eta(t,x)}
\exp\!\left(
-i\int_0^x v(t,y)\,dy + i \Theta(t)
\right) = \Phi\big((\eta(t,x),v(t,x),\Theta(t)\big)
\]
is a solution to~\eqref{NLS}
in $C^0\big((-T_{\max},T_{\max});\mathcal{NX}(\mathbb{R})\big)$.
\end{prop}

\subsection{Estimate in a neighborhood of a chain}

The hydrodynamical and classical variables $Q_c=(\eta_c,v_c)$ are related to the travelling wave by the formula
    \begin{equation}\label{formule de gu_c classique en fonction eta_c et v_c hydrodynamique}
        \gu_c(x)=\sqrt{1-\eta_c(x)}e^{-i\int_0^x v_c(r)dr}.
    \end{equation}

By Remark~\ref{remarque: travelling wave onde progressive en 0}, we also deduce that $\max_\R\eta_c = \eta_c(0)$ and in particular that $\eta_c'(0)=0$.
From~\cite{Bert2}, we recall that there exist $a_{d},K_{d}>0$ independent of $c\in (c_{\min},c_s)$ and $x\in\R$ such that,
\begin{equation}\label{estimée décroissance exponentielle à tout ordre pour eta_c et v_c}
\sum_{\substack{0\leq k_1\leq 3\\ 0\leq k_2\leq 2}}\Big(|\partial_x^{k_1} \eta_c(x)|+c^{1+2k_2}|\partial_x^{k_2} v_c(x)|\Big)\leq K_{d}e^{-a_{d}\sqrt{c_s^2-c^2} |x|}.
\end{equation}

\begin{rem}\label{définition K_lip}
For any $\gc^*\in (c_{\min},c_s)^N$ and for any closed ball $\mathcal{B}\subset (c_{\min},c_s)^N$ centered in $\gc^*$, there exists a uniform constant denoted by $K_{lip}$, such that for any $\gc\in\mathcal{B}$ and $\ga^*,\ga\in\R^N$,
\begin{equation*}
    \normX{R_{\gc^*,\ga^*}-R_{\gc,\ga}}\leq K_{lip}\Big(|\gc^*-\gc| + |\ga^*- \ga|\Big).
\end{equation*}
\end{rem}

A direct consequence of~\eqref{estimée décroissance exponentielle à tout ordre pour eta_c et v_c} is the following lemma, which is proved in~\cite{Bert2}.
\begin{lem}[\cite{Bert2}]\label{lemme: norm Lii de eta_c^* inférieur à delta}
    Let $\gc^*\in\adm_N$. Then there exists $l_1>0$ and $M_1\in (0,1)$ such that for any $\ga\in\pos_N(l_1)$, we have \begin{equation*} \normLii{\eta_{\gc^*,\ga}}\leq M_1.
    \end{equation*} 
\end{lem}

According to Remark~\ref{remarque: non continuité des application Phi}, the mapping through original and hydrodynamical frameworks is not Lipschitz-continuous. However, we show the topological correspondence of both frameworks, at least in a neighborhood of a decoupled sum of travelling waves.

\begin{lem}\label{lemme: équivalence des distances classique et hydro proche d'une chaine de soliton}
Let $\gc^*\in \adm_N$ and $\ga\in\R^N$. There exists a ball $\mathcal{B}_*\subset \mathcal{X}(\R)$ of radius $r_*$ centered in $R_{\gc^*,\ga,0}^{or}$ for the metric $d$ and two positive constants $A_2,l_2,m_2,M_2$ depending only on $\gc^*$ so that the following holds. Assume that $u\in\mathcal{B}_*$ can be lifted as $u=|u| e^{i\varphi}$ with $m_2\leq \inf_\R |u|\leq\sup_{\R}|u|\leq M_2$ and $\varphi(0)=\varphi_{\gc^*,\ga}(0)$. Then writing $Q=(1-|u|^2,-\varphi')\in \energysethydro$, we have, for any $\ga\in\pos_N(l_2)$
\begin{equation*}
    d\left(u,R_{\gc^*,\ga,0}^{or}\right)\leq A_2(1+|\ga|)\normX{Q-R_{\gc^*,\ga}^{hy}}.
\end{equation*}

Moreover, for any $\Theta\in\R$,
\begin{equation}\label{controle norme X par distance d}
    \normX{Q-R_{\gc^*,\ga}^{hy}}\leq A_2(1+|\ga|)d\left(u,R_{\gc^*,\ga,\Theta}^{or}\right).
\end{equation}

\end{lem}

\begin{proof}
Let us proceed with each term in $d$ separately. The term $\normLdeux{|u|^2 - |R_{\gc^*,\ga,0}^{or}|^2}$ can be directly bounded by $\normX{Q-R_{\gc^*,\ga}^{hy}}$. Regarding the local $L^{\ii}$-norm, we invoke Lemma~\ref{lemme: norm Lii de eta_c^* inférieur à delta}, which leads to $\inf_\R |\gu_{\gc^*,\ga}|\geq 1-m_1>0$, and by Lipschitz continuity of $x\mapsto e^{ix}$, and $\varphi(0)=\varphi_{\gc^*,\ga}(0)$, we deduce, for $x\in [-1,1]$, that
\begin{align*}
    \Big|u(x)-R^{or}_{\gc^*,\ga,0}(x)\Big| &\leq \Big||u(x)|-|R^{or}_{\gc^*,\ga,0}(x)|\Big| + |R_{\gc^*,\ga,0}^{or}(x)| \left|e^{i\varphi(x)}- e^{i\varphi_{\gc^*,\ga}(x)}\right|\\
    &\leq \dfrac{ \Big||u(x)|^2-|R^{or}_{\gc^*,\ga,0}(x)|^2\Big| }{ |u(x)|+|R^{or}_{\gc^*,\ga,0}(x)| } + |R^{or}_{\gc^*,\ga,0}(x)|\left(|\varphi(0)-\varphi_{\gc^*,\ga}(0)| + \int_0^x \left|v- v_{\gc^*,\ga^* }\right| \right)\\
    &\leq \dfrac{\normLii{\eta-\eta_{\gc^*,\ga}}}{1-m_1}+ |\gu_{\gc^*,\ga}(x)|\sqrt{|x|}\normLdeux{v^{(N)}-v_{\gc^*,\ga^* }}.
\end{align*}

Hence, bounding roughly $|R_{\gc^*,\ga^*}^{or}(x)\sqrt{|x|}|$ by a constant $A_2>0$, we can assume that 
\begin{equation*}
    \Vert u-R^{or}_{\gc^*,\ga^*,0}(t)\Vert_{L^\ii([-1,1])}\leq A_2\normX{Q-R^{hy}_{\gc^*,\ga^*}(t)}
\end{equation*}

Finally, using again $\varphi(0)=\varphi_{\gc^*,\ga}(0)$, we compute
\begin{align}
    \normLdeux{u' - (R_{\gc^*,\ga}^{or})'}&\leq \normLdeux{\dfrac{\eta'}{2\sqrt{1-\eta}}-\dfrac{\eta_{\gc^*,\ga}'}{2\sqrt{1-\eta_{\gc^*,\ga}}}} + \normLdeux{\varphi'\sqrt{1-\eta}-\varphi_{\gc^*,\ga}'\sqrt{1-\eta_{\gc^*,\ga}}}\notag\\
    &\quad \quad  +\normLdeux{\left(\dfrac{\eta_{\gc^*,\ga}'}{2\sqrt{1-\eta_{\gc^*,\ga}}}+i\varphi_{\gc^*,\ga}'\sqrt{1-\eta_{\gc^*,\ga}}\right)(e^{i\varphi}- e^{i\varphi_{\gc^*,\ga}})}\notag\\
    &\leq \dfrac{\normLii{\eta_{\gc^*,\ga}'}}{2}\normLdeux{\dfrac{1}{\sqrt{1-\eta_{\gc^*,\ga}}}-\dfrac{1}{\sqrt{1-\eta}}} + \normLdeux{\eta_{\gc^*,\ga}'-\eta'}\normLii{\dfrac{1}{\sqrt{1-\eta}}}\label{estimée u'}\\
    &\quad \quad +\normLii{\sqrt{1-\eta}}\normLdeux{v-v_{\gc^*,\ga}} + \normLii{v_{\gc^*,\ga}}\normLdeux{\sqrt{1-\eta} - \sqrt{1-\eta_{\gc^*,\ga}}}\notag\\
    &\quad \quad +\normLdeux{\sqrt{|x|}\left(\dfrac{\eta_{\gc^*,\ga}'}{2\sqrt{1-\eta_{\gc^*,\ga}}}+i\varphi_{\gc^*,\ga}'\sqrt{1-\eta_{\gc^*,\ga}}\right)}\normLdeux{v- v_{\gc^*,\ga}}\notag.
\end{align}

For $\iota\in\{-\frac{1}{2},\frac{1}{2}\}$, $x\mapsto|x|^\iota$ is Lipschitz continuous on any complement set of a ball centered in $0$, and more precisely, we have 
\begin{equation*}
    \left|(1-\eta_{\gc^*,\ga})^\iota-(1-\eta)^\iota\right|\leq\dfrac{|\eta-\eta_{\gc^*,\ga}|}{\min(1-m_1,m_2)^{1-\iota }},
\end{equation*}

hence, using the exponential decay~\eqref{estimée décroissance exponentielle à tout ordre pour eta_c et v_c}, and up to taking a larger constant $A_*$, we deduce from~\eqref{estimée u'} that
\begin{align*}
    \normLdeux{u' - (R_{\gc^*,\ga}^{or})'}&\leq \dfrac{K_d}{2\min(1-m_1,m_2)^\frac{3}{2}}\normLdeux{\eta_{\gc^*,\ga}-\eta} + \dfrac{1}{\sqrt{m_2}}\normLdeux{\eta_{\gc^*,\ga}'-\eta'} \\
    &\quad + M_2\normLdeux{v_{\gc^*,\ga}-v} + \dfrac{K_d}{\min(1-m_1,m_2)^\frac{1}{2}\min_{j\in\{1,...,N\}}|c_j^*|}\normLdeux{\eta_{\gc^*,\ga}-\eta}\\
    &\quad +K_d\left(\dfrac{1}{2\sqrt{m_2}}+\dfrac{\sqrt{M_2}}{\min_{j\in\{1,...,N\}}|c_j^*|}\right)\sum_{j=1}^N \normLdeux{e^{-\sqrt{c_s^2- (c_j^*)^2}|x- a_j|}\sqrt{|x|}} \normLdeux{v_{\gc^*,\ga}-v} \\  
    &\leq A_2 \left(1 + \sum_{j=1}^N \normLdeux{\sqrt{|x|}e^{-\sqrt{c_s^2- (c_j^*)^2}|x- a_j|}}\right)\normX{Q - Q_{\gc^*,\ga}}.
\end{align*}

It remains to give a precise estimate of the quantity
\begin{equation}\label{terme a controler linéaire en a}
    \int_\R |x|e^{-2\sqrt{c_s^2- (c_j^*)^2}|x- a_j|}dx,
\end{equation}
which is linearly controlled in $|\ga|=\max_{j\in\{1,...,N\}}|a_j|$, with a constant only depending on $\gc^*$.

Now, we follow step by step the proof of Lemma A.1 in~\cite{Bert3} to prove the converse inequality~\eqref{controle norme X par distance d}. The only difference lies in the estimate of the term $\normLdeux{\sqrt{|x|}R^{or}_{\gc^*,\ga,\Theta}}$, but can be controlled linearly in $\ga$, similarly to~\eqref{terme a controler linéaire en a}.

\end{proof}

\section{Asymptotic stability of a well-prepared chain}\label{section: Asymptotic stability of a well-prepared chain}

In this section, we prove the asymptotic stability of a well-prepared chain of solitons, i.e. Theorem~\ref{théorème: stab asymp classical of a chain}. The proof of this result will be partitioned into two parts, one dedicated to the region taking into account the whereabouts of the solitons, and another between each soliton. In view of the special configuration of such a well-prepared chain, the interaction of each soliton with one another will become all the more negligible, hence the first part can be reduced to the same argument as for the asymptotic stability of a single travelling wave. To do this, we switch to the hydrodynamical framework, and we prove the asymptotic stability of a well-prepared chain, in this setting, namely Theorem~\ref{théorème: stab asymp classical of a chain hydro} below. In Subsection~\ref{section: From orbital to asymptotic stability}, we state a preliminary result on a weaker form of stability of a well-prepared chain of solitons, crucial for the asymptotic stability. In Subsection~\ref{section: Proof of the intermediate results}, we state a Liouville-type result, which implies the proof of Theorem~\ref{théorème: stab asymp classical of a chain hydro} under some decay properties of the limiting profiles constructed in Subsection~\ref{section: From orbital to asymptotic stability}. Subsection~\ref{section: Proof of the algebraic decay} is dedicated to proving the previous decay properties, through a monotonicity formula. Finally in Subsection~\ref{section: Asymptotic stability in the original framework}, we show how the preceding subsections articulate and imply Theorem~\ref{théorème: stab asymp classical of a chain hydro} and then we recover the asymptotic stability of a chain in the original setting and conclude the proof of Theorem~\ref{théorème: stab asymp classical of a chain}.

\subsection{From orbital to asymptotic stability}\label{section: From orbital to asymptotic stability}

The proof of the asymptotic stability result is based on the orbital stability of a well-prepared chain of solitons in the hydrodynamical setting, that was proved in~\cite{Bert2}. This will eventually lead us to the asymptotic stability of a hydrodynamical chain of solitons.

\begin{thm}[\cite{Bert2}]\label{théorème: stabilité orbitale}
    Let $\gc^*\in\adm_N$. Assume that~\eqref{hypothèse de croissance sur F minorant intermediaire},~\eqref{hypothèse de croissance sur F majorant} and~\eqref{condition suffisante pour la stabilité orbitale sur f''(1)+6f'(1)>0} hold. There exist positive constants $\alpha_*,l_*,A_*,\tau_*,\iota_* >0$, such that the following holds. If $Q_0=(\eta_0,v_0)\in\Nenergysethydro$ is such that for some $\ga^0:=(a_1^0,...,a_N^0)\in \pos_N (L_0)$ with $L_0\geq l_*$,
\begin{equation*}
    \beta_0:=\normX{Q_0-R^{hy}_{\gc^*,\ga^0}}\leq \alpha_*,
\end{equation*}

then, the unique solution $Q(t)=\big(\eta(t),v(t)\big)$ to~\eqref{NLShydro} associated with the initial datum $(\eta_0,v_0)$ is globally defined and there exists $(\ga,\gc)\in C^1(\R_+,\R^{2N})$ such that the function $\varepsilon=(\varepsilon_\eta,\varepsilon_v)$ defined by the formula $Q(t)=R^{hy}_{\gc(t),\ga(t)} + \varepsilon(t)$ satisfies for any $t\in\R_+$,
\begin{equation}\label{borne uniforme dans stabilité orbitale}
    \normX{\varepsilon(t)}+\normR{ \gc(t)-\gc^*}\leq A_* K(\beta_0,L_0),
\end{equation}

with $K(\beta_0,L_0)= \beta_0+e^{-\tau_* L_0 }$. Regarding the modulation parameters, we have
\begin{equation}\label{controle uniforme parametre modulation stab orbitale}
    \normR{ \ga'(t)-\gc(t)}^2+\normR{ \gc'(t)}\leq A_*\Big(\normX{\varepsilon(t)}^2 +  e^{-\tau_*L_0}\Big).
\end{equation}

Moreover, for any $t\in\R$, \begin{equation}\label{proposition: borne uniforme sur les quantités nu_c etc}
    \min_{j\in\{1,...,N\}}\left\{c_j(t),\sqrt{c_s^2 - c_j(t)^2},\inf_\R(1-\eta_{c_j(t)}),\inf_\R \big(1-\eta(t)\big)\right\}\geq \iota_*.
\end{equation}

\end{thm}

\begin{rem}
    Since every norm on a finite dimensional Banach space are topologically equivalent, we use, here as in the sequel, one unique notation to designate every norm on $\R^N$ or $M_N(\R)\sim\R^{N^2}$. In fact, we write $\normR{\mathfrak{x}}$ for any $\mathfrak{x}:=(x_1,...,x_N)\in\R^N$. 
\end{rem}

We now elaborate on how orbital stabil:ity implies the asymptotic stability. We shall provide the missing proofs in the sequel. Assume that initially $\normX{Q_0-R_{\gc^*,\ga^0}^{hy}}= \beta_0\leq \beta_*$ with $\ga^0\in\pos_N(L_0)$ for some $L_0\geq L_*$ and $0<\beta_0\leq \beta_*$, with $\beta_*\leq\alpha_*$ and $L_*\geq l_*$ to be fixed later. Recall that by definition $\varepsilon(t)=\big(\varepsilon_\eta(t),\varepsilon_v(t)\big)=Q(t)-R^{hy}_{\gc(t),\ga(t)}$, then for $j\in\{1,...,N\}$,
    \begin{equation*}
        \varepsilon\big(t,.+a_j(t)\big)=Q\big(t,.+a_j(t)\big)-Q_{c_j(t)} - \sum_{k\neq j}Q_{c_k(t)}\big(. + a_j(t)-a_k(t)\big).
    \end{equation*}

Define $\xi_j(t)=Q(t)-Q_{c_j(t),a_j(t)}$.
By Theorem~\ref{théorème: stabilité orbitale} and the exponential decay property~\eqref{estimée décroissance exponentielle à tout ordre pour eta_c et v_c}, one can recover\footnote{We refer to Proposition~1.14 in~\cite{Bert2} for more details.} that for $k\neq j, |a_j(t)-a_k(t)|\longrightarrow +\ii$ as $t\rightarrow +\ii$, thus 
\begin{equation*}
    Q_{c_k(t)}\big(. + a_j(t)-a_k(t)\big)\ttendii 0\quad\text{in }\mathcal{X}_{hy}(\R).
\end{equation*}

Therefore, there exist a sequence of time $(t_n)$ tending to $+\ii$, a speeds $\gc^\ii=(c_1^\ii,...,c_N^\ii)\in\R^N$ and a limit profiles $\widetilde{\xi}_{j,0}\in\energysethydro$ such that for any $j\in\{1,...,N\}$, 
\begin{equation*}
    \xi_j\big(t_n,.+a_j(t_n)\big)\ntendfX \widetilde{\xi}_{j,0},
\end{equation*}
and
\begin{equation}\label{gc(t_n) tend vers c^ii}
    \gc(t_n)\ntend \gc^\ii.
\end{equation}

The stability of a chain then completely reduces to the case of one single travelling wave. We can in particular rely on~\cite{Bert3} and use a Liouville-type property, which yields $\widetilde{\xi}_{j,0}\equiv 0$, hence
\begin{equation*}
    \varepsilon\big(t_n,.+a_j(t_n)\big)\ntendfX 0.
\end{equation*}

For the part between the travelling waves, we consider a function $A_j$ focusing on the area between the travelling waves of speeds $c_{j-1}^\ii$ and $c_{j}^\ii$ associated with the convergence~\eqref{gc(t_n) tend vers c^ii}. Again, by~\eqref{borne uniforme dans stabilité orbitale}, there exists a sequence of time $(t_n)$ tending to $+\ii$ and a limit profile $\widetilde{\varepsilon}_{j,0}$ such that 
\begin{equation}\label{weak convergence varepsilon vers limit profile}
    \varepsilon\big(t_n,.+A_j(t_n)\big)\ntendfX \widetilde{\varepsilon}_{j,0}.
\end{equation}

Now we consider the solution $\widetilde{\varepsilon}_j=(\widetilde{\eta}_j,\widetilde{v}_j)$ associated with the initial data $\widetilde{\varepsilon}_j(0):=\widetilde{\varepsilon}_{j,0}$, which satisfies by~\eqref{borne uniforme dans stabilité orbitale} and~\eqref{weak convergence varepsilon vers limit profile}, 
\begin{equation}
    \normX{\widetilde{\varepsilon}_{j,0}}\leq \liminf_{n\rightarrow +\ii}\normX{\varepsilon(t_n)}\leq A_* K(\beta_0,L_0).
\end{equation}

Up to decreasing the values of $\alpha_*$ and up to increasing the values of $l_*$ and $A_*$ and stemming from the orbital stability of the trivial solution (see Proposition~\ref{prop: orbital stability of trivial solution hydro} below), that is $Q\equiv (0,0)$, we shall deduce the fact that $\widetilde{\varepsilon}_j$ is globally defined and for any $t\in\R$, \begin{equation*}
    \normX{\widetilde{\varepsilon_j}(t)}\leq A_*K(\beta_0,L_0).
\end{equation*}

Then, we rely on a Liouville-type property on solutions close to the trivial solution and algebraically uniformly decaying in time, implying $\widetilde{\varepsilon}_j(t)\equiv 0$.

The asymptotic stability result of a chain of well-prepared solitons in the hydrodynamical setting reads as follows.
\begin{thm}\label{théorème: stab asymp classical of a chain hydro}
     Under the same assumptions of Theorem~\ref{théorème: stab asymp classical of a chain}. Let $\gc^*\in \adm_N$. There exist $\beta_{*},L_* >0$ such that for any initial data $Q_0\in\energysethydro$ satisfying \begin{equation*}
        \normX{Q_0-\Rhy_{\gc^*,\ga^*}}\leq \beta_{*}\quad\text{with }\ga^*\in\pos_N(L_*),
    \end{equation*}
    then $Q(t)$ the solution associated with the initial condition $Q_0$ is globally defined and there exist $\gc^\ii\in (c_1,c_s)^N $ and two functions $(\gc,\ga)\in C^1(\R_+,\R^2)$ such that the following holds. Set
    \begin{equation}\label{décomposition orthogonale en hydro}
        \varepsilon(t)=Q(t)-\Rhy_{\gc(t),\ga(t)},
    \end{equation}
    then for any $j\in\{1,...,N\}$,
    \begin{equation}\label{convergence of varepsilon(t,a_j(t)) to zero}
    \varepsilon\big(t,.+a_j(t)\big)\tiitendfX 0,
    \end{equation}

    and for any $j\in\{1,...,N\}$,
    \begin{equation}\label{convergence of a_j'(t)) to c_j^ii}
        a_j'(t)\ttendii c^\ii_j\quad\quad\text{ and }\quad\quad c_j(t)\ttendii c^\ii_j.
    \end{equation}

Moreover for translation parameters $\gA=(A_1,...,A_{N+1})\in C^0(\R_+,\R^{N+1})$ satisfying for any $j\in\{1,...,N\}$ and $t\in\R$,
\begin{equation}\label{hypothèse sur A_j entre a_j et a_j+1}
    A_j(t)<a_j(t)<A_{j+1},
\end{equation}
and
\begin{equation}\label{hypothèse sur A_j'}
     \lim_{t\rightarrow +\ii}A'_j(t)=:\gamma_j\in (c_{j-1}^\ii,c_j^\ii),
\end{equation}

then there exists a sequence $(t_n)$ tending to $+\ii$ such that for any $j\in\{1,...,N+1\}$,
 \begin{equation}\label{varepsilon tend vers 0 A_j(t)}
        \varepsilon\big( t_n, .+A_j(t_n)\big)\ntendfX 0.
    \end{equation}

\end{thm}

\begin{rem}
    The fact that the convergence~\eqref{varepsilon tend vers 0 A_j(t)} holds only along a subsequence of time is sufficient to prove the convergences in Theorem~\ref{théorème: stab asymp classical of a chain}. For the sake of convenience, we state the convergence along a subsequence of time in the previous result, and recover the stronger convergence along $t\rightarrow +\ii$ in Subsection~\ref{section: Asymptotic stability in the original framework}.
\end{rem}

As stipulated before Theorem~\ref{théorème: stab asymp classical of a chain hydro}, we shall only prove that a solution taken initially close to a sum of travelling wave converges weakly to zero when it is along the functions $A_j(t)$, that is when it is moving between two consecutive solitons $Q_{c_{j-1}^\ii}$ and $Q_{c_j^\ii}$. Note that in this framework, the assumptions on $A_j$, namely~\eqref{hypothèse sur A_j entre a_j et a_j+1} and~\eqref{hypothèse sur A_j'} are a refined version of conditions~\eqref{B_k(t)<b_k(t)<B_k+1(t),} and~\eqref{c_k-1^ii < liminf_trightarrow +iidfracB_k(t)}. We first prove the asymptotic stability under the former assumptions, and we prove that the result remains true under the latter. The following subsections are dedicated to the proof of the fact that $\widetilde{\varepsilon}_j\equiv 0$ and that the convergence does not depend on the subsequence $(t_n)$.

\subsection{Rigidity property between the travelling waves}\label{section: Proof of the intermediate results}

First, we state the orbital stability of the trivial solution in the hydrodynamical variables. This result relies on the fact that the distance of a solution $Q$ to the trivial solution for the topology of the energy set is almost equal to the energy of $Q$. This result was established in the original setting (see Proposition 1.2 in~\cite{Gerard2}) and an analogue statement also holds in the hydrodynamical framework.

\begin{prop}\label{prop: orbital stability of trivial solution hydro}
There exists $K,\alpha>0$ such that if $Q_0=(\eta_0,v_0)\in\mathcal{NX}_{hy}(\R)$ satisfies $\normX{Q_0}\leq \alpha$, then the solution $Q(t)$ associated with $Q_0$ is globally defined and satisfies $\normX{Q(t)}\leq K\alpha$.
\end{prop}

\begin{proof}
Let $Q(t)$ the solution associated with the initial data $Q_0$ and $T_{\max}$ the maximal time of existence given by Theorem~\ref{théorème: local global well-posedness of cauchy problem}. Taking $\alpha$ small enough, so that the $L^\ii$-norm of $\eta$ is small (by the one-dimensional Sobolev embedding), and in view of the expression of the energy~\eqref{expression de l'énergie en hydro}, we can show that there exists a constant $C$ such that $Q_0\in\Nenergysethydro,\frac{1}{C}E(Q_0)\leq \normX{Q_0}\leq C E(Q_0)<1$. By conservation of the energy in the local well-posedness result, we can implement a continuation argument and show that for any $t\in [0,T_{\max}),CE\big(Q(t)\big)<1$ and $\normLii{\eta(t)}\leq C E\big(Q(t)\big)\leq C\alpha<1$. Then by characterization of the maximal time of existence $T_{\max}$, $Q$ is globally defined and $CE\big(Q(t)\big)\leq C\alpha$ is valid for any time $t\in\R$.

\end{proof}

The following rigidity property is sufficient to prove that $\widetilde{\varepsilon}_j\equiv 0$.
\begin{prop}\label{prop: rigidity property entre les soliton, version algébrique}
There exist $\beta_*,L_*>0$ such that the solution $\widetilde{\varepsilon}_j=(\widetilde{\eta}_j,\widetilde{v}_j)$ to~\eqref{NLShydro}, corresponding to the initial data $\widetilde{\varepsilon}_{j,0}$ satisfying
\begin{equation}\label{borne sur normx eps}
    \normX{\widetilde{\varepsilon}_{j,0}}\leq A_* K(\beta_0,L_0).
\end{equation}
is globally defined.\\
Furthermore, there exist an integer $l_0\geq 2$, such that the following holds. If there exist $r > 5/2, T_{\min}>0,C>0$ and a function $\widetilde{A}_j$ satisfying~\eqref{hypothèse sur A_j entre a_j et a_j+1} and~\eqref{hypothèse sur A_j'} such that for any $(t,x)\in [T_{\min},+\ii)\times \R$ and any integer $l\in\{0,...,l_0\}$, we have
\begin{equation}\label{lemme: decaying property in rigidity}
    \left|\partial_x^{l+1}\widetilde{\eta}_j\big(t,x+\widetilde{A}_j(t)\big)\right|+\left|\partial_x^{l}\widetilde{\eta}_j\big(t,x+\widetilde{A}_j(t)\big)\right|+\left|\partial_x^{l}\widetilde{v}_j\big(t,x+\widetilde{A}_j(t)\big)\right|\leq \dfrac{C}{1+|x|^r},
\end{equation}
then for any $t\in\R$,\begin{equation*}
    \widetilde{\varepsilon}_j(t)\equiv 0.
\end{equation*}
\end{prop}

\begin{proof}[Proof of Proposition~\ref{prop: rigidity property entre les soliton, version algébrique}]

First of all, by the condition~\eqref{lemme: decaying property in rigidity}, $\widetilde{\varepsilon}_{j,0}\in \mathcal{NX}_{hy}^4(\R)$, hence by Theorem~\ref{théoreme: LWP of NLShydro}, we know that the maximal time of existence does not depend on the regularity of $f$ and then $\widetilde{\varepsilon}_j$ exists locally in $\mathcal{NX}_{hy}^4(\R)$. Moreover, up to taking $\alpha_*$ and $1/L_*$ small enough, by Proposition~\ref{prop: orbital stability of trivial solution hydro} and~\eqref{borne sur normx eps}, the solution exists globally and in particular by~\eqref{NLShydro}, we deduce that $\widetilde{\varepsilon}_j\in C^0\big(\R,\mathcal{NX}_{hy}^4(\R)\big)\cap C^1\big(\R,\mathcal{NX}_{hy}^2(\R)\big)$ and satisfies for any $t\in\R$, 
\begin{equation}\label{borne sur varepsilon AK}
    \normX{\widetilde{\varepsilon_j}(t)}\leq A_*K(\beta_0,L_0).
\end{equation}

By smoothness of the previous quantities, we can linearize~\eqref{NLShydro} around the trivial solution $(\eta,v)=(0,0)$. This leads to the equation
\begin{equation}\label{expression de partial_t varepsilon}
    \partial_t\widetilde{\varepsilon}_j = JL(\widetilde{\varepsilon}_j) + JN(\widetilde{\varepsilon}_j),
\end{equation}
with
\begin{equation}\label{expression de J}
    \begin{array}{l}
    JQ=-2S\partial_x Q \\
\end{array}
\quad\text{with}\quad S:=\begin{pmatrix}
        0 & 1\\
        1 & 0
    \end{pmatrix},
\end{equation}

\begin{equation*}
    L(\widetilde{\varepsilon}_j)=\begin{pmatrix}
        \dfrac{1}{4}\left(-\partial_x^2 \widetilde{\eta}_j + c_s^2\widetilde{\eta}_j\right)\\
        \widetilde{v}_j
    \end{pmatrix},
\end{equation*}
and
\begin{equation*}
    N(\widetilde{\varepsilon}_j)=\begin{pmatrix}
        N_1(\widetilde{\varepsilon}_j)\\
        N_2(\widetilde{\varepsilon}_j)
    \end{pmatrix}=\begin{pmatrix}
        \dfrac{\widetilde{\eta}_j^2}{2}\displaystyle\int_0^1(1-s)f''(1-s\widetilde{\eta}_j)ds -\widetilde{v}_j^2 - \dfrac{(\partial_x\widetilde{\eta}_j)^2}{4(1-\widetilde{\eta}_j)^2}+\dfrac{\widetilde{\eta}_j\partial_x^2\widetilde{\eta}_j}{2}\int_0^1\dfrac{ds}{(1-s\widetilde{\eta}_j)^2}\\
        \widetilde{\eta}_j\widetilde{v}_j
    \end{pmatrix}.
\end{equation*}

We conclude the argument by considering the following virial quantity 
\begin{equation*}
    V(t)=\int_\R \mu(x)\widetilde{\eta}_j(t,x)\widetilde{v}_j(t,x)dx,
\end{equation*}
where we have set $\mu(x)=x - \widetilde{A}_j(t)$, and which shall provide the suitable estimate to control the perturbation $\widetilde{\varepsilon}_j$. We shall use the previous linearization in a proper way, taking care of the regularity of each quantity and eventually derive the next lemma.

\begin{lem}\label{lemme: borne inférieure du viriel}
The virial function $V$ is well-defined and differentiable on $[T_{\min},+\ii)$. Up to increasing the value of $T_{\min}$, $V$ is uniformly bounded and there exist two constants $C,p>0$ such that for $t\geq T_{\min}$,
    \begin{align*}
    V'(t)\geq C\normX{\widetilde{\varepsilon}_j(t)}^2-\grandOde{K(\beta_0,L_0)^p\normX{\widetilde{\varepsilon}_j(t)}^2}.
\end{align*}
\end{lem}

If we admit Lemma~\ref{lemme: borne inférieure du viriel}, then the values of $\beta_*$ and $1/L_*$ can be reduced sufficiently so that for a potentially smaller constant that we keep denoting by $C>0$, we have
$V'(t)\geq \frac{C}{2}\normX{\widetilde{\varepsilon}_j(t)}^2$. We infer from the previous lemma that $V$ is increasing and is bounded on $[T_{\min},+\ii)$, thus there exists $l_+\in\R$ such that $V(t)\longrightarrow l_+$ as $t\rightarrow +\ii$ and then 
\begin{equation*}
    \int_{T_{\min}}^\ii\normX{\widetilde{\varepsilon}_j(t)}dt = l_+ - V(T_{\min}) < +\ii.
\end{equation*}

By finiteness, we can exhibit a sequence $(t_n)$ tending to $+\ii$ such that $\normX{\widetilde{\varepsilon}_j(t_n)}$ tends to $0$. Thus, using the orbital stability result conversely in time, there is no other choice than $\widetilde{\varepsilon}_j(t)\equiv 0$.
\end{proof}

It remains to prove the virial lemma.
\begin{proof}[Proof of Lemma~\ref{lemme: borne inférieure du viriel}]

We verify that $V$ is well-defined and differentiable. Using the bound~\eqref{lemme: decaying property in rigidity}, $V$ is well-defined. Now, writing $V(t)=\frac{1}{2}\psLdeuxLdeux{\mu S\widetilde{\varepsilon}_j}{\widetilde{\varepsilon}_j}$, we must check that each term in the right-hand side of the expression of $\partial_t\widetilde{\varepsilon}_j$, that is~\eqref{expression de partial_t varepsilon}, can be integrated against the function $\mu S \widetilde{\varepsilon}_j$. We have for instance
\begin{align*}
    \left|\psLdeuxLdeux{\mu S JL\widetilde{\varepsilon}_j}{\widetilde{\varepsilon}_j}\right| =\left|2\psLdeuxLdeux{ \partial_x L\widetilde{\varepsilon}_j}{\mu\widetilde{\varepsilon}_j}\right| \leq 2\norm{\widetilde{\varepsilon}_j}{\mathcal{X}_{hy}^2}\normLdeuxLdeux{\mu\widetilde{\varepsilon}_j}.
\end{align*}

Using~\eqref{lemme: decaying property in rigidity}, we infer that the previous term is bounded. As for the term involving $N(\widetilde{\varepsilon}_j)$, we bound uniformly the term involving the nonlinearity $f$ by a constant and we similarly deduce that
\begin{align*}
    \left|\psLdeuxLdeux{\mu S JN\widetilde{\varepsilon}_j}{\widetilde{\varepsilon}_j}\right| \lesssim\norm{\widetilde{\varepsilon}_j}{\mathcal{X}_{hy}^2}\norm{\widetilde{\varepsilon}_j}{\mathcal{X}_{hy}}\normLdeuxLdeux{\mu\widetilde{\varepsilon}_j}.
\end{align*}

As a consequence, we deduce that $V$ is differentiable, so that we can write
\begin{align*}
    V'(t) &= W_1(t) + W_2(t),
\end{align*}
with 
\begin{align*}
    W_1(t) = \psLdeuxLdeux{\mu S JL\widetilde{\varepsilon}_j}{\widetilde{\varepsilon}_j} -\widetilde{A}_j'(t)\psLdeux{\widetilde{\eta}_j}{\widetilde{v}_j}\quad\text{and}\quad W_2(t) = \psLdeuxLdeux{\mu S JN\widetilde{\varepsilon}_j}{\widetilde{\varepsilon}_j}.
\end{align*}

Proceeding some integrations by part, and by the standard Young inequality, we obtain
\begin{align*}
    W_1(t) &=-2\psLdeuxLdeux{\partial_x L\widetilde{\varepsilon}_j}{\widetilde{\varepsilon}_j}-\widetilde{A}_j'(t)\psLdeux{\widetilde{\eta}_j}{\widetilde{v}_j}\\
    &=\dfrac{1}{2}\int_\R\big(\partial_x\widetilde{\eta}_j(t)\big)^2 +  \int_\R \big(\widetilde{v}_j(t)\big)^2 +\dfrac{c_s^2}{2}\int_\R\big(\widetilde{\eta}_j(t)\big)^2 - \widetilde{A}_j'(t)  \int\widetilde{\eta}_j (t)\widetilde{v}_j(t)\\
    &\geq \dfrac{1}{2}\int_\R\big(\partial_x\widetilde{\eta}_j(t)\big)^2 +\dfrac{c_s^2-\big(\widetilde{A}_j'(t)\big)^2}{2}\int_\R\big(\widetilde{\eta}_j(t)\big)^2  + \dfrac{1}{2}\int\big(\widetilde{v}_j(t)\big)^2,
\end{align*}

Now, we use~\eqref{hypothèse sur A_j'} and infer that for any $j\in\{1,...,N\},\ \gamma_j< c_s$, so that for $T_{\min}$ larger, there exists a constant $C\in (0,\frac{1}{2})$, only depending on $\gc^\ii$, such that for $t\geq T_{\min},\ c_s^2-\big(\widetilde{A}_j'(t)\big)^2\geq C$. As a consequence, we deduce that for $t\geq T_{\min},\ W_1(t)\geq C\normX{\widetilde{\varepsilon}_j(t)}^2$.

On the other hand,
\begin{align*}
    W_2(t) &= 2\int_\R \mu\partial_x\widetilde{v}_j \widetilde{v}_j \widetilde{\eta}_j + 2\int_\R\widetilde{\eta}_j\widetilde{v}_j^2 + 2\int_\R N_1(\widetilde{\varepsilon}_j)\widetilde{\eta}_j + 2\int_\R \mu N_1(\widetilde{\varepsilon}_j)\partial_x\widetilde{\eta}_j .
\end{align*}

We have $\int_\R \mu\partial_x\widetilde{v}_j \widetilde{v}_j \widetilde{\eta}_j=\grandOde{\normLii{\mu\partial_x\widetilde{\eta}_j}\normLdeux{\widetilde{v}_j}^2}$ and $2\int_\R\widetilde{\eta}_j\widetilde{v}_j^2 = \grandOde{\normLii{\widetilde{\eta}_j}\normLdeux{\widetilde{v}_j}^2}$. Using the one-dimensional Sobolev inequality, and~\eqref{borne sur varepsilon AK}, we derive $2\int_\R\widetilde{\eta}_j\widetilde{v}_j^2=\grandOde{K(\beta_0,L_0)\normX{\widetilde{\varepsilon}_j}^2}$. Regarding $\normLii{\mu\partial_x\widetilde{\eta}_j}$, we obtain, using~\eqref{lemme: decaying property in rigidity}, that
\begin{align*}
    \normLii{\mu\partial_x\widetilde{\eta}_j}\leq \normLii{\mu\sqrt{|\partial_x\widetilde{\eta}_j|}}\sqrt{\normLii{\partial_x\widetilde{\eta}_j}}=\grandOde{\sqrt{\Vert\widetilde{\eta}_j\Vert_{H^2}}}.
\end{align*}

We can integrate by part and deduce $\Vert\widetilde{\eta}_j\Vert_{H^2}^2\leq \normX{\widetilde{\varepsilon}_j}\Vert\widetilde{\eta}_j\Vert_{H^3}$. Therefore, using~\eqref{lemme: decaying property in rigidity} and~\eqref{borne sur varepsilon AK}, we derive
\begin{equation*}
    \int_\R \mu\partial_x\widetilde{v}_j \widetilde{v}_j \widetilde{\eta}_j=\grandOde{\big(K(\beta_0,L_0)\big)^{\frac{1}{4}}\normX{\widetilde{\varepsilon}_j}^2}.
\end{equation*}

Next, we proceed to an integration by parts with the term involving a derivative to the second order in $\widetilde{\eta}_j$ in $N_1(\widetilde{\varepsilon}_j)$. Then, by potentially reducing the value of $K(\beta_0,L_0)$ in~\eqref{borne sur varepsilon AK}, so that the terms involving the nonlinearity $f$ are bounded, we deduce that $2\int_\R N_2(\widetilde{\varepsilon}_j)\widetilde{\eta}_j = \grandOde{K(\beta_0,L_0)\normX{\widetilde{\varepsilon}_j}^2}$. In addition, by another integration by part on the higher order term, we have $2\int_\R \mu N_1(\widetilde{\varepsilon}_j)\partial_x\widetilde{\eta}_j =\grandOde{\normLii{\mu\partial_x\widetilde{\eta}_j}\normX{\widetilde{\varepsilon}_j}^2}$, that can be dealt with as previously.

\end{proof}

\subsection{Proof of the algebraic decay}\label{section: Proof of the algebraic decay}

To derive the algebraic decay \eqref{lemme: decaying property in rigidity} on the limiting profile $\widetilde{\varepsilon}_j=(\widetilde{\eta}_j,\widetilde{v}_j)$, we shall first extract some estimate on $Q(t)$ from a monotonicity formula (see Proposition \ref{proposition: monotonie formula} below). Once we proved the monotonicity formula on $Q$ and obtained estimate of the form~\eqref{monotonie troisieme}, the fact that estimates of the type \eqref{lemme: decaying property in rigidity} hold for $\widetilde{\varepsilon}_j$ is reminiscent from the method in \cite{Bert3} and due to the weak continuity of the flow of~\eqref{NLShydro} (see Proposition~\ref{proposition: weak continuity of the flow in hydrodynamical setting} in~\cite{Bert3}).

First, we prove that a very generic solution to~\eqref{NLShydro}, taken in the vicinity of travelling waves, or of the constant solutions of modulus one, namely the trivial solutions, enjoys decaying properties, that are sufficient for fulfilling the assumption~\eqref{lemme: decaying property in rigidity} in the rigidity theorem.

We consider a cut-off version of the momentum. For general $Q=(\eta,v)\in\energysethydro$, we define
\begin{equation}\label{définition localized momentum}
    p_{R}(Q)=\dfrac{1}{2}\int_\R \eta v  \chi_{R},
\end{equation}
where $\chi_{R}(x)=\chi(x-R)$ and
\begin{equation}\label{définition fonction chi}
    \chi(x)=\dfrac{1}{2}\Big( 1 + \tanh\left(\frac{\tau}{2} x\right)\Big),
\end{equation}

for a constant $\tau>0$ to be fixed later. We write
\begin{equation}\label{définition gamma_*}
    \gamma_* = \min_{j\in\{1,...,N\}}\big(\gamma_j,c_j^\ii-\gamma_j\big),
\end{equation}

\begin{comment}

\begin{rem}\label{remarque: widetilede a_j ou a_j dans la preuve de la monotonie}

\textcolor{red}{DIRE AUSSI UN MOT sur le fait que les parametres sont $\widetilde{a}_j(t)$ dans la decomposition orthogonale, et qu'on peut le faire comme si c'était les $a_j(t)$, respectant ainsi les conditions~\ref{convergence of a_j'(t)) to c_j^ii} et \eqref{hypothèse sur A_j entre a_j et a_j+1}, peut etre besoin d'une proposition comme Proposition 3.2 dans Yakine}
\end{rem}
\end{comment}

where $\gamma_j$ is given in~\eqref{hypothèse sur A_j'}. To obtain a uniform algebraic control on the original solution $Q$, we rely on the monotonicity formula and then recover the desired estimate on $\widetilde{\varepsilon}_j$ in a second time.
\begin{prop}\label{proposition: monotonie formula}
There exist positive constants $\kappa_{\min},\tau_{\min},\sigma_{\min},T_{\min}$ only depending on $\gc^*$ such that for any $t\in [T_{\min},+\ii)$, any $j\in\{1,...,N+1\}$ and any $(R,\sigma)\in\R\times  [-\sigma_{\min},\sigma_{\min}]$,
\begin{equation}\label{monotonie principal}
    \dfrac{d}{dt}\Big[p_{R+\sigma t + A_j(t)}\big(Q(t)\big)\Big]\geq \kappa_{\min}\int_\R\Big( (\partial_x\eta(t))^2 + \eta(t)^2 + v(t)^2\Big)\chi_{R+\sigma t + A_j(t)}' + \grandOde{e^{-\tau_{\min} |R+\sigma t|}}.
\end{equation}

In particular, there exists a constant that we also denote by $\kappa_{\min}>0$ such that for $T_{\min}\leq t_1\leq t_2$ and $R'\in\R$, we have 
\begin{equation}\label{monotonie secondaire}
    p_{j,R' + A_j(t_2)}(t_2)\geq p_{j,R'+A_j(t_1)}(t_1) -\kappa_{\min} e^{-\tau_{\min}|R'|} .
\end{equation}

and 
\begin{equation}\label{monotonie troisieme}
    \int_{t}^{t+1}\int_\R \Big((\partial_x\eta)^2 + \eta^2 + v^2\Big)\big(s,x+A_j(s)\big)\chi'_R(x) dx ds = \grandOde{e^{-\tau_{\min} R}}.
\end{equation}
\end{prop}

\begin{rem}\label{remarque: relaxation des hypotheses pour monotonie}

A similar monotonicity formula has been proven in~\cite{Bert3} (see Proposition~1.25). However in the latter framework, the solution decomposes as a single travelling wave plus an extra perturbation, and the estimates hold for the momentum, while focusing \textit{around} the soliton. Now, instead of considering general functions $A_j$ located between the solitons, we can specifically restrict ourselves to $X_j(t) = \frac{a_{j+1}(t)+a_j(t)}{2}$ for $j\in\{1,...,N-1\}$. In this case, provided that the solution $Q$ is initially close to a well-prepared chain of solitons, then an analogous monotonicity formula still holds. Indeed, an adaptation of the proof of the previous monotonicity formula inspired from the one in~\cite{Bert2} leads us to 
\begin{equation*}
    \dfrac{d}{dt}\Big[p_{ X_j(t)}\big(Q(t)\big)\Big]\geq -\kappa_{\min} e^{-\tau_{\min}t}.
\end{equation*}
\end{rem}

Due to the weak continuity of the flow under~\eqref{NLShydro}, we obtain improved convergence properties for the flow of~\eqref{NLShydro}, as well as for the modulation parameters. 
\begin{lem}\label{lemme: continuité faible du flow le long de la suite t_n}
Let $t\in\R$ and $j\in\{1,...,N\}$ be fixed. There exists $\widetilde{A}_j$ satisfying~\eqref{hypothèse sur A_j entre a_j et a_j+1} and~\eqref{hypothèse sur A_j'} such that 

    \begin{equation*}
        Q\big(t_n+t,.+A_j(t_n)\big)\ntendfX\widetilde{\varepsilon}_j(t),
    \end{equation*}

    \begin{equation*}
        Q\big(t_n+t,.+A_j(t_n+t)\big)\ntendfX\widetilde{\varepsilon}_j\big(t,.+\widetilde{A}_j(t)\big),
    \end{equation*}
    and
    \begin{equation*}
        A_j(t_n+t)-A_j(t_n)\ntend \widetilde{A}_j(t)
    \end{equation*}
\end{lem}

\begin{proof}[Proof of Lemma~\ref{lemme: continuité faible du flow le long de la suite t_n}]
The convergences above were already proven in Subsection A.1.1 in~\cite{BetGrSm2}, in the Gross-Pitaevskii case. Nonetheless, it does not depend on the nonlinearity considered but only on the weak-continuity of the hydrodynamic flow, i.e. Proposition~\ref{proposition: weak continuity of the flow in hydrodynamical setting}. It can just be adapted to this framework, by setting $\widetilde{A}_j(t):=\gamma_j t$, which also satisfies~\eqref{hypothèse sur A_j entre a_j et a_j+1} and~\eqref{hypothèse sur A_j'}. We refer to~\cite{BetGrSm2} for more details.

\end{proof}

A straightforward consequence of Proposition~\ref{proposition: monotonie formula} and Lemma~\ref{lemme: continuité faible du flow le long de la suite t_n} is that if we pass to the weak limit (and then integrate in time), we recover a similar estimate for the limit profile, that is
\begin{equation*}
    \int_{t}^{t+1}\int_\R \Big((\partial_x\widetilde{\eta}_j)^2 + \widetilde{\eta}_j^2 + \widetilde{v}_j^2\Big)\big(s,x+\widetilde{A}_j(s)\big)\chi'_R(x) dx ds = \grandOde{e^{-\tau_* R}}.
\end{equation*}

In view of the assumptions on the nonlinearity $f$, we can argue as in the proof of Proposition 1.27 in~\cite{Bert3}, so that we obtain the desired estimate
\begin{equation*}
    \left|\partial_x^{l+1}\widetilde{\eta}_j\big(t,x+\widetilde{A}_j(t)\big)\right|+\left|\partial_x^{l}\widetilde{\eta}_j\big(t,x+\widetilde{A}_j(t)\big)\right|+\left|\partial_x^{l}\widetilde{v}_j\big(t,x+\widetilde{A}_j(t)\big)\right|\leq \dfrac{C}{1+|x|^r}.
\end{equation*}

~~\\

It remains to prove the monotonicity formula.
\begin{proof}[Proof of Proposition~\ref{proposition: monotonie formula}]
The proof can be adapted from the proof of Proposition~1.25 in~\cite{Bert3}. As in the latter, we derive the expression
\begin{align*}
    \dfrac{d}{dt}\left(p_{R+\sigma t+A_j(t)}(t)\right) = &\int_\R \pi_{R+\sigma t+A_j(t)}(\eta(t),v(t)),
\end{align*}

where we set $\widetilde{F}(\rho)=\rho f(1-\rho) - F(1-\rho)$ and
\begin{align*}
    \pi_{R+\sigma t+A_j(t)}(\eta,v)&=\dfrac{1}{2}\chi'_{R+\sigma t+A_j(t)}\bigg((1-2\eta)v^2 + \widetilde{F}(\eta)+\dfrac{(3-2\eta)(\partial_x\eta)^2}{4(1-\eta)^2} -\big(A_j'(t)+\sigma\big) \eta v\bigg)\\
    &\quad +\dfrac{1}{4}\chi'''_{R+\sigma t+A_j(t)}\big(\eta+\ln(1-\eta)\big).
\end{align*}

Similarly to the method used in~\cite{Bert3}, we show using assumption~\eqref{hypothèse de croissance sur F minorant intermediaire}, that there exists a constant $C>0$, only depending on $\gc^*$, such that
\begin{equation}\label{controle sur la quantité pi_R}
    \dfrac{1}{C}\chi_R'\Big((\partial_x\eta)^2+  q(\eta,v)\Big)\leq \pi_R(\eta,v),
\end{equation}

where \begin{align*}
    q(\eta,v)&=(1-2\eta)v^2-\mu_*|\eta v| +\left(-\int_0^1 rf'(1-r\eta)dr -\tau^2 C_{\ln}\right)\eta^2 .
\end{align*}

In fact, the only difference with the proof in~\cite{Bert3} lies in the estimate of the coefficient associated with the cross term $\eta v$. From~\eqref{hypothèse sur A_j'}, we infer that
\begin{equation*}
    \big| A_j'(t)+\sigma|\leq \big|A_j'(t)-\gamma_j\big| + |\gamma_j| + |\sigma|\leq \big|A_j'(t)-\gamma_j\big| + c_j^\ii + \sigma_{\min}.
\end{equation*}

Using~\eqref{convergence of a_j'(t)) to c_j^ii}, we can enlarge the value of $T_{\min}>0$ and take $\sigma_{\min}>0$ small enough so that for $t\geq T_{\min}$, \begin{equation*}
    \mu_*:=\sup_{t\geq T_{\min}}\big\{\big|A_j'(t)-l_j\big| + c_j^\ii+ \sigma_{\min}\big\} <c_s.
\end{equation*}

We now separate the real line into two disjoint parts that are an interval $I_j(t)$ and its complement. The interval $I_j(t)$, centered in $A_j(t)$, focuses away from the solitons, where the solution $\varepsilon$ is small and the quadratic form $q$ is positive. On the other hand, $I_j(t)^c$ takes into account the whereabout of the solitons but it will be controlled by using the exponential decay of the cut-off function $\chi$. We set 
\begin{equation*}
    I_j(t):=\left[A_j(t)-\dfrac{R_0}{2},A_j(t) + \dfrac{R_0}{2}\right],
\end{equation*}

where $R_0>0$ is to be fixed later in the following claim.

\begin{claim}\label{coercivité dans la forme quad dans la monotonie du moment}
There exists $\tau_{\min}>0$ such that for any $\tau\leq\tau_{\min}$ and possibly shrinking the value of $\beta_*$ and $1/L_*$, and up to increasing the values of $R_0,T_{\min},\kappa_{\min}$ such that for $(t,\sigma)\in [T_{\min},+\ii)\times [-\sigma_{\min},\sigma_{\min}]$, we have on $I_j(t)$:
    \begin{equation*}
        q(\eta,v)\geq \kappa_{\min} \big(\eta^2+v^2\big).
    \end{equation*}
\end{claim}

We split the integral as announced previously, and up to enlarging the value of $C$ we obtain
\begin{align*}
    \dfrac{d}{dt}\big(\widetilde{p}_{R+\sigma t+A_j(t)}(t)\big)&\geq\dfrac{1}{C}\int_{I_j(t)}\Big((\partial_x\eta)^2+\eta^2 + v^2\Big)\big(t,.+A_j(t)\big)\chi'_{R+\sigma t}\\
    &\quad\quad\quad+ \int_{I_j(t)^c}\pi_{R+\sigma t+A_j(t)}(\eta(t),v(t))\\
    &\geq \dfrac{1}{C}\int_{\R}\Big((\partial_x\eta)^2+\eta^2 + v^2\Big)\big(t,.+A_j(t)\big)\chi'_{R+\sigma t} \\
    &\quad\quad\quad+ \int_{I_j(t)^c}\left(\pi_{R+\sigma t+A_j(t)}(\eta(t),v(t)) -\dfrac{\big((\partial_x\eta)^2+\eta^2 + v^2\big)\chi'_{R+\sigma t+A_j(t)}}{C}\right) 
\end{align*}

Furthermore, using~\eqref{hypothèse de croissance sur F minorant intermediaire} and~\eqref{proposition: borne uniforme sur les quantités nu_c etc}, we have the useful bounds
\begin{equation}\label{controle des termes du moment par l'énergie}
    (\partial_x\eta)^2\leq \dfrac{2(\partial_x\eta)^2}{(1-\eta)},\quad \eta^2\leq \dfrac{4}{c_s^2}\dfrac{F(1-\eta)}{2}\quad\text{and}\quad v^2\leq \dfrac{(1-\eta)v^2}{\iota_*},
\end{equation}

hence, taking $\kappa_{\min}$ even larger, we can assume that on $\R$,
\begin{equation*}
    \big|\pi_{R + \sigma t+A_j(t)}(\eta,v) - \dfrac{1}{C}\big((\partial_x\eta)^2+\eta^2 + v^2\big)\chi'_{R+\sigma t+A_j(t)}\big|\leq \kappa_{\min}e(\eta,v)\big|\chi'_{R+\sigma t+A_j(t)}\big|,
\end{equation*}

where $e(\eta,v)$ is the energy density defined in~\eqref{expression de l'énergie en hydro}. Since we have for any $x\in I_j(t)^c$,
\begin{equation*}
    \chi'_{R+\sigma t+A_j(t)}(x)\leq \tau_{\min} e^{-\tau_{\min}|x-(R+\sigma t+A_j(t))|}\leq \tau_{\min} e^{\frac{\tau_{\min} R_0}{2}}e^{-\tau_{\min}|R+\sigma t|}=\grandOde{e^{-\tau_{\min}|R+\sigma t|}},
\end{equation*}

as well as $\chi'''_{R+\sigma t+A_j(t)}(x)\leq \grandOde{e^{-\tau_{\min}|R+\sigma t|}}$, we obtain the control
\begin{align*}
    \dfrac{d}{dt}\big(\widetilde{p}_{R+\sigma t+A_j(t)}(t)\big)\geq -\kappa_{\min} E(\eta,v)e^{-\tau_{\min}|R+\sigma t|} + \kappa_{\min}\int_\R \Big((\partial_x\eta)^2+\eta^2 + v^2\Big)\big(t,.+A_j(t)\big)\chi'_{R+\sigma t},
\end{align*}
which concludes the proof of the main estimate by conservation of the energy.

To obtain the almost monotonicity formula for times $t_1\leq t_2$, we just apply the previous estimate. For $R\geq 0$, we first integrate between $t_1$ and $(t_1+t_2)/2$ with a special choice of parameters $\sigma =\sigma_{\min}$ and $R=R'-\sigma_{\min} t_2$. Then, we integrate between $(t_1+t_2)/2$ and $t_2$ with $\sigma =-\sigma_{\min}$ and $R=R'+\sigma_{\min} t_2$. The case $R\leq 0$ can be obtained similarly.

\end{proof}

It remains to prove the claim.
\begin{proof}[Proof of Claim~\ref{coercivité dans la forme quad dans la monotonie du moment}]

We write \begin{align*}
    q(\eta,v)&=\delta_1(\eta) \eta^2 - \delta_2|\eta v| + \delta_3 (\eta)v^2,
\end{align*}
with 
\begin{align*}
    & \delta_1(\eta) = -\int_0^1 rf'(1-r\eta)dr -\tau^2 C_{\ln},\\
    &\delta_2=\mu_* ,\\
    &\delta_3(\eta)=1-2\eta,
\end{align*}

and obtain the expression
\begin{equation*}
    q(\eta,v)=\widetilde{\delta}_{1,1}\left(|\eta|-\dfrac{\delta_2}{2\widetilde{\delta}_{1,1}}|v|\right)^2 + \widetilde{\delta}_{1,2}(\eta)\eta^2 + \left(\delta_3(\eta)-\dfrac{\delta_2^2}{4\widetilde{\delta}_{1,1}}\right)v^2,
\end{equation*}
where we have set $\widetilde{\delta}_{1,1}=\frac{\mu_*^2+c_s^2}{4}$, and $\widetilde{\delta}_{1,2}(\eta):=\delta_1(\eta)-\widetilde{\delta}_{1,1}$. We recall that $\mu_*<c_s$, so that $\widetilde{\delta}_{1,1}<\frac{c_s^2}{2}$. In addition, we set $\gd:=\frac{c_s^2}{2(\mu_*^2+c_s^2)}$. As a consequence of the preceding choices, by the orthogonal decomposition of $\varepsilon$ in~\eqref{décomposition orthogonale en hydro}, then the exponential decay~\eqref{estimée décroissance exponentielle à tout ordre pour eta_c et v_c} and~\eqref{proposition: borne uniforme sur les quantités nu_c etc}, we have for any $(t,x)\in\R\times I_j(t)$,
\begin{equation*}
    |\eta(t,x)|\leq \sum_{k=1}^N|\eta_{c_k(t),a_k(t)}(x)|+|\varepsilon_\eta(t,x)|\leq K_d \sum_{k=1}^Ne^{-a_d \iota_* |x-a_k(t)|}+\normHun{\varepsilon_\eta}.
\end{equation*}

Furthermore, for $x\in I_j(t)$, and any $k\in\{1,...,N\}$, we have
\begin{align*}
    |x-a_k(t)|&\geq |a_k(t)-A_j(t)| - |x-A_j(t)|\geq a_j(t)-A_j(t) - \dfrac{R_0}{2}\\
    &\geq \int_{T_{\min}}^t\big(a_j'(s)-A_j'(s)\big)ds  +a_j(T_{\min})-A_j(T_{\min})-\dfrac{R_0}{2},
\end{align*}

and
\begin{align*}
    \int_{T_{\min}}^{t}\big(a_j'(s)-A_j'(s)\big)ds &=\int_{T_{\min}}^{t}\big(a_j'(s)-c_j(s)\big)ds +\int_{T_{\min}}^{t}\big(c_j(s)-c_j^\ii\big)ds + (t-T_{\min})(c_j^\ii - \gamma_j)\\
    &\quad\quad+ \int_{T_{\min}}^{t}\big(\gamma_j-A_j'(s)\big)ds .
\end{align*}

We fix the value $T_{\min}$ large enough and and shrink sufficiently the values of $\beta_*$ and $1/L_*$, so that we have by~\eqref{borne uniforme dans stabilité orbitale},~\eqref{controle uniforme parametre modulation stab orbitale},~\eqref{hypothèse sur A_j'} and~\eqref{définition gamma_*},
\begin{equation*}
    \int_{T_{\min}}^{t}\big(a_j'(s)-A_j'(s)\big)ds \geq \dfrac{\gamma_*}{2}(t-T_{\min}).
\end{equation*}

Finally, we set $R_0 = a_j(T_{\min})-A_j(T_{\min})$, and we have shown that for any $(t,x)\in [T_{\min},+\ii)\times I_j(t)$,
\begin{equation*}
    |x - a_k(t)|\geq \dfrac{\gamma_*}{2}(t-T_{\min}) + \dfrac{R_0}{2}.
\end{equation*}

hence,
\begin{equation*}
    \normLii{\eta(t)}\leq K_d e^{-\iota_*\big(\frac{\gamma_*}{2}(t-T_{\min}) + \frac{R_0}{2}\big)}+A_*K(\beta_0,L_0)
\end{equation*}

Then we can assume that $|\eta| \leq \frac{\gd}{2}$ and since $-2f'(1-\xi)$ tends to $c_s^2$ as $\xi$ tends to $0$, there exists $\tau_1>0$ such that if $\tau\leq\tau_1 $, we have
\begin{equation*}
    \quad\widetilde{\delta}_{1,1}< \delta_1 (\eta)<\dfrac{c_s^2}{2}.
\end{equation*}

This implies that $\delta_3(\eta)-\frac{\delta_2^2}{4\widetilde{\delta}_{1,1}}\geq \gd$ and $\widetilde{\delta}_{1,2} \geq \frac{c_s^2}{8}$ which provides the suitable constant $\kappa_*$.
\end{proof}

\subsection{Asymptotic stability in the original framework}\label{section: Asymptotic stability in the original framework}

Using the notations of Subsection~\ref{section: From orbital to asymptotic stability}, we finish the proof of Theorem~\ref{théorème: stab asymp classical of a chain hydro}.

\begin{proof}[Proof of Theorem~\ref{théorème: stab asymp classical of a chain hydro}]
Regarding properties~\eqref{décomposition orthogonale en hydro},~\eqref{convergence of varepsilon(t,a_j(t)) to zero} and~\eqref{convergence of a_j'(t)) to c_j^ii}, we refer to the discussion in Subsection~\ref{section: From orbital to asymptotic stability}. As for the part between each soliton, Subsection~\ref{section: Proof of the algebraic decay} is dedicated to showing that the assumptions in the rigidity property, namely Proposition~\ref{prop: rigidity property entre les soliton, version algébrique}, are achieved. We take $\beta_*$ and $1/L_*$ small enough such that the results of the latter subsection are true, then Proposition~\ref{lemme: decaying property in rigidity} applies and the profile limit is identically equal to zero which concludes the proof of Theorem~\ref{théorème: stab asymp classical of a chain hydro} under assumptions~\eqref{hypothèse sur A_j entre a_j et a_j+1} and~\eqref{hypothèse sur A_j'}.
\end{proof}

Now, we proceed to the proof of Theorem~\ref{théorème: stab asymp classical of a chain} in several parts. First, we show that the convergence~\eqref{varepsilon tend vers 0 A_j(t)} in Theorem~\ref{théorème: stab asymp classical of a chain hydro} still holds under looser assumptions on the translation parameters $A_j$. In a second part, we prove that the convergences actually do not depend on the extraction and is true as $t\rightarrow +\ii$. Finally, we recover the asymptotic stability in the original framework.

\begin{proof}[Proof of Theorem~\ref{théorème: stab asymp classical of a chain}]

\underline{Step 1: Relaxation of the translation parameter $A_j$.} We assume that $A_j$ satisfies~\eqref{hypothèse sur A_j entre a_j et a_j+1} and
\begin{equation}\label{hypothèse sur A_j des liminf}
    c_{j-1}^\ii < \liminf_{t\rightarrow +\ii}\dfrac{A_j(t)}{t}\leq \limsup_{t\rightarrow +\ii}\dfrac{A_j(t)}{t}<c_j^\ii,
\end{equation}
instead. Therefore, using~\eqref{hypothèse sur A_j des liminf} and up to extracting a subsequence, we have
\begin{equation*}
    \lim_{n\rightarrow +\ii}\dfrac{A_j(t_n)}{t_n}=:\gamma_j^\ii\in (c_{j-1}^\ii,c_j^\ii).
\end{equation*}

Since we consider only the limit along the sequence $(t_n)$ and since $A_j$ satisfies~\eqref{hypothèse sur A_j entre a_j et a_j+1} and~\eqref{hypothèse sur A_j des liminf}, we shall construct for any $j$ a smooth function satisfying~\eqref{hypothèse sur A_j entre a_j et a_j+1} and tending to $\gamma_j^\ii$ as $t\rightarrow +\ii$. 

Since $(t_n)$ tends to $+\ii$, we can assume that for any integer $n$, we have $t_{n+1}\geq 2t_n $. It is sufficient to set a new function $\widetilde{A}_j$ defined in the following way. For any integer $n$, we prescribe the value of $\widetilde{A}_j$ in the vicinity of $t_n$, by imposing $\widetilde{A}_j(t)=A_j(t_n) + \dfrac{A_j(t_{n+1})-A_j(t_n)}{t_{n+1}-t_n}(t-t_n)$. Then we glue all parts of this function in a smooth way so that $\widetilde{A}_j\in C^1(\R_+)$. We verify that for any integer $n$, $\widetilde{A}_j(t_n)=A_j(t_n)$ and
\begin{align*}
    \widetilde{A}_j'(t_ n)=\dfrac{A_j(t_{n+1})-A_j(t_n)}{t_{n+1}-t_n} = \dfrac{A_j(t_{n+1})}{t_{n+1}}+ \left(\dfrac{A_j(t_{n+1})}{t_{n+1}}-\dfrac{A_j(t_{n})}{t_{n}}\right)\dfrac{t_n}{t_{n+1}-t_n}\ntend \gamma_j^\ii.
\end{align*}

Therefore, by~\eqref{hypothèse sur A_j des liminf}, the function $\widetilde{A}_j$ can be set to satisfy
\begin{equation*}
    \lim_{t\rightarrow +\ii}\widetilde{A}'_j(t)=\lim_{n\rightarrow +\ii}\dfrac{A_j(t_n)}{t_n}=\gamma_j^\ii.
\end{equation*}
as well as the assumptions~\eqref{hypothèse sur A_j entre a_j et a_j+1} and~\eqref{hypothèse sur A_j'}, thus the results of the previous sections apply. Namely, the limit profiles $\widetilde{\varepsilon}_j$ fulfills the assumptions in Proposition~\ref{lemme: decaying property in rigidity} and thus by~\eqref{weak convergence varepsilon vers limit profile}, 
\begin{equation}\label{weak convergence varepsilon vers limit profile en widetile A_j}
    \varepsilon\big(t_n,.+\widetilde{A}_j(t_n)\big)\ntendfX 0.
\end{equation}

Then by construction of $\widetilde{A}_j$,
\begin{equation}
\varepsilon\big(t_n,.+A_j(t_n)\big)\ntendfX 0.
\end{equation}

\underline{Step 2: The convergence does not depend on the subsequence.} 

We can show that for any another subsequence $(s_n)$ tending to $+\ii$, the sequence $\big(\varepsilon(s_n,.+A_j(s_n))\big)_n$ will always be bounded in the hydrodynamical energy set, by orbital stability. Assume that, by contradiction
\begin{equation}
\varepsilon\big(s_n,.+A_j(s_n)\big)\ntendfX \varepsilon_{j},
\end{equation}
with $\varepsilon_{j} \neq 0$. The arguments of Step 1 can be implemented along the sequence $(s_n)$ and then the profile would still be compliant with the previous properties. This profile is then necessarily identically equal to $0$. This shows that the convergence does not depend on the subsequence and concludes this step.

\underline{Step 3: Convergences in the original setting.}
By Lemma~\ref{lemme: équivalence des distances classique et hydro proche d'une chaine de soliton}, we can fix $\delta_*>0$ small enough so that if $d(\psi_0,R_{\gc^*,\ga^*,\Theta^*}^{or})\leq \delta_*$, then $\normX{Q_0-R_{\gc^*,\ga^*}^{hy}}\leq \beta_*$. Hence Theorem~\ref{théorème: stab asymp classical of a chain hydro} applies and we have the decomposition~\eqref{décomposition orthogonale en hydro} in the hydrodynamical variables. Furthermore, since $c_j(t)$ remains uniformly close to $c_j^\ii$, then by Remark~\ref{définition K_lip}, and convergences~\eqref{convergence of varepsilon(t,a_j(t)) to zero} and~\eqref{convergence of a_j'(t)) to c_j^ii}, we deduce
\begin{align*}
    Q\big(t,.+a_j(t)\big)\tiitendfX Q_{c_j^\ii}.
\end{align*}

Now, for any $j\in\{1,...,N\}$, we implement the same procedure that the one used in Subsection 1.3.4 in~\cite{Bert3}, and this provides a smooth phase modulation $\Theta=(\theta_1,...,\theta_N)$ and a refined translation parameter $b=(b_1,...,b_N)$ such that for any $j\in\{1,...,N\}$,
\begin{equation*}
        e^{-i\theta_j(t)} \psi\big( t, .+b_j(t)\big)\tiitendfd \gu_{c^\ii_j},
    \end{equation*}

and
    \begin{equation*}
        b_j'(t)\ttendii c^\ii_j\quad\quad\text{ and }\quad\quad \theta_j'(t)\ttendii 0,
    \end{equation*}

Finally, for any function $\mathfrak{B}$ satisfying~\eqref{B_k(t)<b_k(t)<B_k+1(t),} and~\eqref{c_k-1^ii < liminf_trightarrow +iidfracB_k(t)}, we can proceed similarly and exhibit a limit profile $\psi_{j,\ii}$ such that,
\begin{equation}
         e^{-i\theta_j(t)}\psi\big( t, .+B_j(t)\big)\tiitendf \psi_{j,\ii}\quad\text{in }H^1_{\loc}(\R).
    \end{equation}

Thus by~\eqref{varepsilon tend vers 0 A_j(t)}, $\eta\big(t,.+B_j(t)\big)$ converges weakly to zero in $H^1(\R)$ as $t$ tends to $+\ii$, and by the Rellich theorem, we deduce that we have the pointwise equality $|\psi_{j,\ii}|=1$. Thus there must exist a function $\varphi_{j,\ii}$ such that $\psi_{j,\ii}=e^{i\varphi_{j,\ii}}$. Then~\eqref{varepsilon tend vers 0 A_j(t)}, we also have that $v\big(t,.+B_j(t)\big)$ converges weakly to zero in $L^2(\R)$ as $t$ tends to $+\ii$. Moreover,
\begin{equation*}
    v\big(t,.+B_j(t)\big) = \dfrac{i\psi\big( t, .+B_j(t)\big).\partial_x \psi\big( t, .+B_j(t)\big)}{\big|\psi\big( t, .+B_j(t)\big)\big|}\tiitendf \dfrac{i\psi_{j,\ii}.\partial_x\psi_{j,\ii}}{|\psi_{j,\ii}|^2} = \varphi_{j,\ii}'\text{ in }L^2_{\loc}(\R).
\end{equation*}

This implies that $\varphi_{j,\ii}$ is equal to a constant that we denote by $e^{i\theta_j^\ii}$, which ends the proof.

\end{proof}

\section{Construction of an asymptotic $N$-soliton}\label{section: Construction of an asymptotic N-soliton}

In this section, we construct the asymptotic $N$-soliton in the hydrodynamical framework, and this will imply Theorem~\ref{théorème: existence d'un asymptotic N soliton like solution classique variable}. In particular, we rely on the orbital stability result stated in Subsection~\ref{section: From orbital to asymptotic stability}, as we shall see in Subsection~\ref{section: Proof of the uniform estimate}. First we state the existence of an hydrodynamical $N$-soliton, and then we prove Theorem~\ref{théorème: existence d'un asymptotic N soliton like solution classique variable} by admitting that the following result is true.

\begin{prop}\label{théorème: existence d'un asymptotic N soliton like solution}
    Let $\gc^*\in (c_{\min},c_s)^N$ and $\ga^*\in\R^N$. There exist a global solution $\mathcal{Q}^{(N)}(t,x)=\big(\eta^{(N)}(t,x),v^{(N)}(t,x)\big)\in C^0\big(\R,\mathcal{NX}_{hy}(\R)\big)$ to~\eqref{NLShydro}, and positive constants $T_0,\tau_0,K_0$ such that for any $t\geq T_0$,
\begin{equation*}
    \normX{\mathcal{Q}^{(N)}(t)-\sum_{k=1}^N Q_{c_k^*}(.-c_k^*t-a_k^*)}\leq K_0 e^{-\tau_0t}.
\end{equation*}

\end{prop}

We deduce directly the existence of the asymptotic $N$-soliton in the original setting, by applying Proposition~\ref{prop:Mohamad_3_6_f}.

\begin{proof}[Proof of Theorem~\ref{théorème: existence d'un asymptotic N soliton like solution classique variable}]
From the existence of an asymptotic $N$-soliton in the hydrodynamical variables stated in Proposition~\ref{théorème: existence d'un asymptotic N soliton like solution}, we construct a solution of~\eqref{NLS} by setting 
\begin{equation*}
    \mathcal{U}^{(N)}(t):=\Phi\big(\mathcal{Q}^{(N)}(t),\Theta(t)\big),
\end{equation*}
where $\Phi$ and $\Theta$ are defined in Proposition~\ref{prop:Mohamad_1_1_f} and \ref{prop:Mohamad_3_6_f}. Furthermore, we define the map $\widetilde{\Phi}:\mathcal{NX}_{hy}(\R)\longrightarrow \mathcal{NX}(\R)$ by $\widetilde{\Phi}(\eta,v):=\sqrt{1-\eta}\ e^{-i\int_0^xv}.$

By the one-dimensional Sobolev inequality we can assume, up to increasing the value of $T_0$, that for $t\geq T_0,\ \ga^*+\gc^*t\in\pos_N(l_1)$ and by Proposition~\ref{théorème: existence d'un asymptotic N soliton like solution}, $$\normLii{\eta^{(N)}(t)-\eta_{\gc^*,\ga^*+\gc^*t}}\lesssim \normX{\mathcal{Q}^{(N)}(t)-R_{\gc^*,\ga^*+\gc^*t}^{hy}}< 1 - M_1,$$
with $l_1$ and $M_1$ from Lemma~\ref{lemme: norm Lii de eta_c^* inférieur à delta}. Thus, by the triangular inequality, we derive $\normLii{\eta^{(N)}(t)} < 1$ for any $t\geq T_0$. As a consequence, 
\begin{equation}
    \inf_{(t,x)\in[T_0,+\ii)\times \R}|\mathcal{U}^{(N)}(t,x)| =m_2>0.
\end{equation}

We can also assume that for any $t\geq T_0, \ga^* + \gc^*t\in\pos_N(l_2)$, so that we can apply Lemma~\ref{lemme: équivalence des distances classique et hydro proche d'une chaine de soliton} with $\ga=\ga^* + \gc^*t,u=\widetilde{\Phi}\big(\mathcal{Q}(t)\big)=\mathcal{U}^{(N)}(t)e^{-i\Theta(t)}$ the phase of which vanishes at $x=0$, as well as the phase of $\widetilde{\Phi}(R^{hy}_{\gc^*,\ga^*+\gc^*t})$. Therefore 
\begin{equation*}
    d\left(\widetilde{\Phi}\big(\mathcal{Q}(t)\big),\widetilde{\Phi}(R_{\gc^*,\ga^* + \gc^*t}^{hy})\right)\leq A_*|\ga^* + \gc^*t|\normX{\mathcal{Q}(t) -R_{\gc^*,\ga^* + \gc^*t}^{hy} }\leq A_*K_0|\ga^* + \gc^*t|e^{-\tau_0t}.
\end{equation*}

Hence, there exists $C_*>0$ such that
\begin{align*}
    d\big(\mathcal{U}^{(N)}(t),R^{or}_{\gc^*,\ga^* + \gc^*t ,\Theta(t)}\big)& = d\big(\mathcal{U}^{(N)}(t)e^{-i\Theta(t)},R^{or}_{\gc^*,\ga^* + \gc^*t ,0}\big) = d\left(\widetilde{\Phi}\big(\mathcal{Q}(t)\big),\widetilde{\Phi}(R_{\gc^*,\ga^* + \gc^*t}^{hy})\right)\\
    &\leq A_* (|\gc^*|t + |\ga^*|) \normX{\mathcal{Q}(t) -R_{\gc^*,\ga^* + \gc^*t}^{hy} } \leq  C_*  e^{-\frac{\tau_*}{2} t}\ttendii 0.
\end{align*}

\end{proof}

\subsection{Reduction of the problem}

We now proceed to the proof of Proposition~\ref{théorème: existence d'un asymptotic N soliton like solution}. Take a sequence $(s_n)$ of time tending to $+\ii$. By Theorem~\ref{théoreme: LWP of NLShydro}, we define for any $n$ the unique solution $Q^n\in C^0\big((s_n-T_{\max,n},s_n+T_{\max,n}),\Nenergysethydro\big)$ of~\eqref{NLShydro} with initial data 
\begin{equation*}
    Q^n(s_n,x)=R^{hy}_{\gc^*,\ga^*+\gc^
* s_n}(x)=\sum_{k=1}^N Q_{c_k^*}(x-c_k^*s_n-a_k^*).
\end{equation*}

Up to rearranging the speeds, we can always assume that $c_1^* < ... <c_N^*$ and then define
\begin{equation}\label{définition sigma_*}
    \sigma_0:=\min\left\{ c_{k+1}^*-c_k^*\ \big| \ k\in\{ 1,...,N-1\}\right\}>0.
\end{equation}
    
We shall prove that there exists a time $s_0$ such that sequence of solutions $\big(Q^n(t)\big)_n$ stays uniformly close to $R^{hy}_{\gc^*,\ga^*+\gc^*t}$ with respect to the time variable $t\in [s_0,s_n]$ and the variable $n$ (taken large enough so that $s_n\geq s_0$). This property is the point of the following proposition.

\begin{prop}\label{proposition: controle uniforme en n de Q^n moins chaine de solitons}
    There exist $s_0\in\R$ and $K_0,\tau_0>0$, such that for any integer $n$ such that $s_n\geq s_0$, the solution $Q^n$ to~\eqref{NLShydro} with prescribed data $Q^n(s_n)=R^{hy}_{\gc^*,\ga^* + \gc^*s_n}$ is in $C^0\big([s_0,s_n],\mathcal{NX}_{hy}(\R)\big)$ and for any $t\in [s_0,s_n]$, we have
    \begin{equation*}
        \Big\Vert Q^n(t)-R^{hy}_{\gc^*,\ga^*+\gc^* t}\Big\Vert_{\mathcal{X}_{hy}}\leq K_0 e^{-\tau_0 t}.
    \end{equation*}
\end{prop}

Now, we admit that Proposition~\ref{proposition: controle uniforme en n de Q^n moins chaine de solitons} holds true and derive a proof of Proposition~\ref{théorème: existence d'un asymptotic N soliton like solution}. We refer to Subsection~\ref{section: Proof of the uniform estimate} for the proof of Proposition~\ref{proposition: controle uniforme en n de Q^n moins chaine de solitons}.

\begin{proof}[Proof of Proposition~\ref{théorème: existence d'un asymptotic N soliton like solution}]
The integer $n$ is fixed such that $s_n\geq s_0$ here. Proposition~\ref{proposition: controle uniforme en n de Q^n moins chaine de solitons} implies that $\big(Q^n(s_0)\big)$ is bounded in $\mathcal{X}_{hy}(\R)$ and there exists a limit profile $\mathcal{Q}_0^{(N)}\in \mathcal{X}(\R)$ such that, up to a subsequence,
\begin{equation*}
    Q^n(s_0)\ntendfX \mathcal{Q}_0^{(N)}.
\end{equation*}

By Theorem~\ref{théoreme: LWP of NLShydro}, we define $\QN=(\eta^{(N)},v^{(N)})\in C^0\big((s_0-T,s_0+T),\Nenergysethydro\big)$ the solution associated with the initial data $\mathcal{Q}^{(N)}(s_0,x)=\QN_0(x)=\big(\eta^{(N)}_0(x),v^{(N)}_0(x)\big)$. 

We write $\ga=\ga^* + \gc^* t$. Up to increasing the value of $s_0$, we have, 
\begin{equation}\label{b_k+1(t) -b_k(t) geq L_0}
    a_{k+1}(t)-a_k(t)\geq\sigma_0 t + a_{k+1}^*-a_k^*:=L_0\geq l_2.
\end{equation}

Using Lemma~\ref{lemme: norm Lii de eta_c^* inférieur à delta} with $\ga\in \pos_N(l_2)$ and Proposition~\ref{proposition: controle uniforme en n de Q^n moins chaine de solitons}, that for any $t\in (s_0-T,s_0+T)\cap [s_0,s_n]$,
\begin{align*}
    \normLii{\eta_n(t)}\leq \normLii{\eta_n(t)-\eta_{\gc^*,\ga^*+\gc^*t}} + \normLii{\eta_{\gc^*,\ga^*+\gc^*t}}<1.
\end{align*}

Thus, we obtain that $Q_n$ is globally defined by Theorem~\ref{théoreme: LWP of NLShydro} and using Proposition~\ref{proposition: weak continuity of the flow in hydrodynamical setting}, we have for any $t\in (s_0-T,s_0+T)$,
\begin{equation*}
    Q^n(t)\ntendfX \mathcal{Q}^{(N)}(t).
\end{equation*}

Therefore, we can also infer that
\begin{equation*}
    Q^n(t)-R^{hy}_{\gc^*,\ga^*+\gc^*t}\ntendfX \mathcal{Q}^{(N)}(t)-R^{hy}_{\gc^*,\ga^*+\gc^*t},
\end{equation*}

hence, by Proposition~\ref{proposition: controle uniforme en n de Q^n moins chaine de solitons}, on $(s_0-T,s_0+T)\cap [s_0,s_n]$,
\begin{equation*}
    \normX{\mathcal{Q}^{(N)}(t)-R^{hy}_{\gc^*,\ga^*+\gc^*t}}\leq \liminf_{n\rightarrow +\ii}\normX{Q^n(t)-R^{hy}_{\gc^*,\ga^*+\gc^*t}}\leq K_0e^{-\tau_0 t}.
\end{equation*}

Thus, up to increasing the value of $s_0$ once more, we can infer similarly that $\mathcal{Q}^{(N)}$ is globally defined, and thus we have constructed an asymptotic $N$-soliton in the hydrodynamical framework.

\end{proof}

\subsection{Proof of the uniform estimate}\label{section: Proof of the uniform estimate}

To prove Proposition~\ref{proposition: controle uniforme en n de Q^n moins chaine de solitons}, we shall implement a continuity method. Relying on localized quantities around the solitons, we prove suitable estimates to draw the argument to a close. This reduction of Proposition~\ref{proposition: controle uniforme en n de Q^n moins chaine de solitons} can be summarized in the following result.

\begin{lem}\label{proposition: reduction proposition uniform estimates}
     There exist $s_0>0$ and $K_0,\tau_0,\alpha_0>0$ satisfying $K_0e^{-\tau_0s_0}\leq\frac{\alpha_0}{2}$ such that the following holds. If for a time $t^*\in [s_0,s_n]$, and any $t\in[t^*,s_n]$, we have
     \begin{equation}\label{proof of uniform estimate: inégalité à améliorer}
         \normX{Q^n(t)-R^{hy}_{\gc^*,\ga^*+\gc^*t}}\leq \alpha_0,
     \end{equation}
     
     then for $t\in[t^*,s_n]$, \begin{equation}\label{inégalité avec A_0/2}
         \normX{Q^n(t)-R^{hy}_{\gc^*,\ga^*+\gc^*t}}\leq K_0e^{-\tau_0t}.
     \end{equation}
\end{lem}

Now, we proceed to the proof of Lemma~\ref{proposition: reduction proposition uniform estimates}.

\begin{proof}[Proof of Lemma~\ref{proposition: reduction proposition uniform estimates}]

Assume that for $t\in[t^*,s_n]$,~\eqref{proof of uniform estimate: inégalité à améliorer} holds. Provided that $\alpha_0\leq \beta_*$ and $s_0$ is large enough so that for any $t\in [t^*,s_n],\ga(t):=\big(a_1(t),...,a_N(t)\big)=\gc^* t +\ga^*\in\pos_N(l_2)$. Therefore, we can construct two modulation parameters $\gc_n(t)=\big(c_{1,n}(t),...,c_{N,n}(t)\big),\gb_n(t)=\big(b_{1,n}(t),...,b_{N,n}(t)\big)\in\pos_N(L_0-1)$ such that for $t\in [s_0,s_n]$,
\begin{equation}\label{orthogonal decomposition en n}
    Q^n(t,x)=\sum_{k=1}^N Q_{c_{k,n}(t)}\big(x-b_{k,n}(t)\big)+\varepsilon^n(t,x),
\end{equation}
and for $k\in\{1,...,N\}$,
\begin{equation}
    \psLdeuxLdeux{\varepsilon^{n}(t)}{\partial_x Q_{c_{k,n}(t)}}=\grad p (Q_{c_{k,n}(t)}).\varepsilon^n(t)=0.
\end{equation}

Moreover, by~\eqref{proof of uniform estimate: inégalité à améliorer} and the fact that $\ga\in\pos_N(l_2)$, we have
\begin{equation}\label{controle que je veux}
    \normX{\varepsilon^n(t)}+\normR{\gc_n(t)-\gc^*}\leq A_*K(\alpha_0,L_0),
\end{equation}
and
\begin{equation}\label{controle des parametres de modulation n}
    \normR{\gb_n'(t)-\gc_n(t)}^2+\normR{\gc_n'(t)}\leq A_*\Big(\normX{\varepsilon^n(t)}^2 + e^{-\tau_*L_0} \Big).
\end{equation}

By uniqueness of the modulation parameters, and in view of the choice $Q^n(s_n)=R^{hy}_{\gc^*,\ga^*+\gc^*s_n}$, we obtain 
\begin{equation}\label{modulation parameters en s_n}
    c_{k,n}(s_n)=c_k^*,\quad a_{k,n}(s_n)=0,\text{ and }\varepsilon^n(s_n)=0.
\end{equation}

To prove~\eqref{proof of uniform estimate: inégalité à améliorer}, we rely on the introduction of the functional
\begin{equation}
    \mathcal{G}(Q):=E(Q)-\sum_{k=1}^N c_k^* q_k(Q), 
\end{equation}
where \begin{equation}\label{définition momentum localisé pour construction}
    q_k\big(Q\big)=q_k( \eta,v)=\int_\R \eta v(\widetilde{\chi}_{k}-\widetilde{\chi}_{k+1}).
\end{equation}

Here we used the following partition of unity. Let $\chi$ be the cut-off function defined in~\eqref{définition fonction chi}, and set
\begin{equation*}
    \widetilde{\chi}_{k}(x)=\left\{
\begin{array}{l}
    1\hspace{28mm}\text{if }k=1, \\
    \chi\big(\widetilde{\tau}(x-X_k(t)\big)\quad\text{if }k\in\{2,...,N\},\quad\text{with }X_k(t):=\dfrac{b_{k+1,n}(t)+b_{k,n}(t)}{2}, \\
    0\hspace{28mm}\text{if }k=N+1.\\
\end{array}
\right.
\end{equation*}

We also define 
\begin{equation*}\label{définition de p_k}
    \widetilde{p}_k(Q):=\widetilde{p}_k(\eta,v):=\dfrac{1}{2}\int_\R \eta v\widetilde{\chi}_k,
\end{equation*}

for some $\widetilde{\tau}>0$, the value of which is fixed later in Claim~\ref{controle de la dérivée de G(t) par exponentielle décroissante}. On the one hand, by construction of the modulation parameters, we obtain some coercivity property for the functional $\mathcal{G}$.
\begin{claim}\label{coercivité pour fonctionnelle de Lyapunov}
    There exist $s_1,\tau_1,K_1,\alpha_1>0$ such that if $s_0\geq s_1,\alpha_0\leq\alpha_1$ and $\widetilde{\tau}+\tau_0\leq \tau_1$, then for $t\in [t^*,s_n]$, we have
\begin{equation*}
    \normX{\varepsilon^n(t)}^2\leq K_1\Big(\mathcal{G}\big(Q^n(t)\big)-\mathcal{G}\big(Q^n(s_n)\big) +\normR{\gc_n(t)-\gc^*}^2+ e^{-\tau_1 t}\Big).
\end{equation*}
\end{claim}

\begin{proof}[Proof of Claim~\ref{coercivité pour fonctionnelle de Lyapunov}]
We gather the estimates from Corollary~1.13 in~\cite{Bert2}, which we adapt to the orthogonal decomposition~\eqref{orthogonal decomposition en n}. Indeed, the estimates of the latter corollary come from Taylor expansions of the energy and a modified momentum with a slightly different cut-off function, however the intrinsic computation do not change in substance and the same type of estimates can be recovered. We impose $\alpha_1$ small enough and $s_1$ large enough to be compliant with the conditions of the latter corollary. We further impose that $\tau_1$ is small enough so that it takes into account the constant $\widetilde{\tau}$ in the corresponding cut-off functions, and so that it absorbs the constants that multiply $t$ in the decreasing exponential term. 

Therefore, using the lower bound in Corollary 1.13 in~\cite{Bert2} for $\mathcal{G}\big(Q^n(t)\big)$, there exists $K_{1,1}$ such that 
\begin{equation*}
    \mathcal{G}\big(Q^n(t)\big) - \sum_{k=1}^N\big(E(Q_{c_k^*})-c_k^*p(Q_{c_k^*})\big) \geq \dfrac{1}{K_{1,1}}\normX{\varepsilon^n(t)}^2 -K_{1,1}\left( \normX{\varepsilon^n(t)}^3 + \normR{\gc_n(t)-\gc^*}^2 + e^{-\tau_1 t}\right) .
\end{equation*}

Moreover, we use the upper bound in Corollary 1.13 in~\cite{Bert2} for $\mathcal{G}\big(Q^n(s_n)\big)$, taking~\eqref{modulation parameters en s_n} into account. Thus, there exists a constant $K_{1,2}$ such that
\begin{equation*}
    \mathcal{G}\big(Q^n(s_n)\big)-\sum_{k=1}^N\big(E(Q_{c_k^*})-c_k^*p(Q_{c_k^*})\big)\leq K_{1,2}e^{-\tau_1 t}
\end{equation*}

Combining both previous estimates, and possibly reducing the value of $\alpha_1$ so that $\normX{\varepsilon^n(t)}$ is small enough, we deduce the existence of a constant $K_1$ such that for any $t\in [t^*,s_n]$,
\begin{equation}
    \normX{\varepsilon^n(t)}^2\leq K_1\Big(\mathcal{G}\big(Q(t)\big)-\mathcal{G}\big(Q(s_n)+\normR{\gc_n(t)-\gc^*}^2+e^{-\tau_1 t} \Big).   
\end{equation}

\end{proof}

On the other hand, refining the monotonicity formula obtained in Proposition~\ref{proposition: monotonie formula}, we have an upper control.
\begin{claim}\label{controle de la dérivée de G(t) par exponentielle décroissante}
    There exist $s_2\geq s_1,0<\tau_2\leq\min(\tau_1,\tau_*\sigma_*),K_2\geq K_1,0<\alpha_2\leq\alpha_1$ such that if $s_0\geq s_2,\alpha_0\leq\alpha_2$ and $\widetilde{\tau}+\tau_0\leq \tau_2$, then we have, for $t\in [t^*,s_n]$,
\begin{equation*}
    \mathcal{G}\big(Q^n(s_n)\big)-\mathcal{G}\big(Q^n(t)\big)\leq K_2 e^{-\tau_2 t}.
\end{equation*}
\end{claim}

\begin{proof}[Proof of Claim~\ref{controle de la dérivée de G(t) par exponentielle décroissante}]

As explained in Remark~\ref{remarque: relaxation des hypotheses pour monotonie}, this result is reminiscent from Proposition~\ref{proposition: monotonie formula}. The momentum $\widetilde{p}_k$ that we currently use differs from the one defined in~\eqref{définition localized momentum} only by the constant $\widetilde{\tau}$. Taking $\alpha_2$ and $1/s_2$ small enough implies by~\eqref{proof of uniform estimate: inégalité à améliorer} and the fact that $\gb_n(t)\in\pos_N(L_0-1)$ that the monotonicity formula in Remark~\ref{remarque: relaxation des hypotheses pour monotonie} holds true. Then, integrating the formula in Remark~\ref{remarque: relaxation des hypotheses pour monotonie} between $t$ and $s_n$ yields
\begin{equation}\label{estimate monotonie sur p_k}
    -\Big(\widetilde{p}_k\big(Q^n(s_n)\big)-\widetilde{p}_k\big(Q^n(t)\big)\Big)\leq \dfrac{\kappa_{\min}}{\tau_{\min}}e^{-\tau_{\min}t},
\end{equation}

where the parameter $\widetilde{\tau}$ is very small compared to $\tau_{\min}$. We set $\tau_2\leq \tau_{\min}$ and $$K_2=\frac{(N-1)\kappa_{\min}\max_{k\in\{1,...,N-1\}}(c_{k+1}^*-c_k^*)}{\tau_{\min}}.$$

By proceeding to some Abel-type rearrangement and by conservation of the energy along the flow of~\eqref{NLShydro}, we infer 
\begin{equation}\label{développement de G entre Q(t) et Q(s_n)}
    \mathcal{G}\big(Q^n(t)\big)-\mathcal{G}\big(Q^n(s_n)\big)=-\sum_{k=1}^{N-1}(c_{k+1}^*-c_k^*)\Big(\widetilde{p}_k\big(Q^n(s_n)\big)-\widetilde{p}_k\big(Q^n(t)\big)\Big)\leq K_2e^{-\tau_2t}.
\end{equation}

\end{proof}

Combining Claim~\ref{controle de la dérivée de G(t) par exponentielle décroissante} and Claim~\ref{coercivité pour fonctionnelle de Lyapunov}, we derive
\begin{equation}
    \normX{\varepsilon^n(t)}^2\leq K_1(1+K_2)e^{-\tau_2t} + K_1\normR{\gc_n(t)-\gc^*}^2.
\end{equation}

In view of~\eqref{controle des parametres de modulation n}, we obtain
\begin{align*}
    |\gc_n(t)-\gc^*|&=\int^{s_n}_t|\gc_n'(s)|ds\leq A_*\int^{s_n}_t\big(\normX{\varepsilon^n(s)}^2 + e^{-\tau_*\sigma_*s}\big)ds\\
    &\leq \dfrac{A_*\big(K_1(K_2+1)+1\big)}{\tau_2}e^{-\tau_2t} + A_*K_1\int_t^{s_n}\normR{\gc_ n(s)-\gc^*}^2ds.
\end{align*}

Reducing possibly the value of $\alpha_2$ and $1/s_2$ so that $K_1A_*^2 K(\alpha_0,L_0)^2\leq 1/2$, and due to~\eqref{controle que je veux}, we infer
\begin{equation*}
    |\gc_n(t)-\gc^*|\leq 2K_3e^{-\tau_2t}\quad\text{with }K_3=\dfrac{A_*\big(K_1(K_2+1)+1\big)}{\tau_2}.
\end{equation*}

We write  $\ga_n(t)=\big(a_{1,n}(t),...,a_{N,n}(t)\big)=\gb_n(t)-\gb(t)$. Now, up to increasing the value of $s_2$, $|\gc_n(t)-\gc^*|$ stays uniformly small by~\eqref{controle que je veux}, up to a proper choice of parameters $\alpha_*$ and $L_*$, so that by definition of the constant $K_{lip}$ in Remark~\ref{définition K_lip}, we have 
\begin{align*}
    \normX{Q^n(t)-R^{hy}_{\gc^*,\ga^*+\gc^*t}}&\leq \normX{\varepsilon^n(t)} +\normX{\sum_{k=1}^N Q_{c_{k,n}(t)}\big(.-c_k^*t-a_k^*-a_{k,n}(t)\big)-R^{hy}_{\gc^*,\ga^*+\gc^*t}}\\
    &\leq \normX{\varepsilon^n(t)} + K_{lip}\big(\normR{\gc_n(t)-\gc^*}+|\ga_n(t)|\big).
\end{align*}

Using in addition~\eqref{controle des parametres de modulation n} yields
\begin{align*}
    \normR{\ga_n(t)}&=\int_t^{s_n}|\ga'_n(s)|ds \leq \int_t^{s_n}\left|\ga'_n(s)+\gb'(s)-\gc_n(s)\right|ds+\int_t^{s_n}\left|\gb'(s)-\gc_n(s)\right|ds\\
    &\leq \int_t^{s_n}\normR{\ga'_n(t)+\gc^*-\gc_n(t)} + \int_t^{s_n}\left|\gc^*-\gc_n(s)\right|ds\\
    &\leq \sqrt{\dfrac{A_*\big(K_1(K_2+1)+1\big)}{\tau_2}} e^{-\frac{\tau_2}{2}t} + \dfrac{A_*\big(K_1(K_2+1)+1\big)}{\tau_2^2}e^{-\tau_2t} .
\end{align*}

If we further impose that $\tau_0\leq \frac{\tau_2}{2}$, 
\begin{equation*}
    K_0\geq  \sqrt{K_1(1+K_2)} +  K_{lip}\left(2K_3+\sqrt{K_3} + \dfrac{K_3}{\tau_2} \right),
\end{equation*}
and $s_0$ such that $K_0e^{-\tau_0s_0}\leq\frac{\alpha_0}{2}$,
then for $t\in [t^*,s_n]$,~\eqref{inégalité avec A_0/2} holds and the lemma is proved.
\end{proof}

We conclude by the proof of Proposition~\ref{proposition: controle uniforme en n de Q^n moins chaine de solitons}.

\begin{proof}[Proof of the Proposition~\ref{proposition: controle uniforme en n de Q^n moins chaine de solitons}]
We first define 
    \begin{equation}
        t^*:= \inf J_n,\quad\text{where } J_n:=\left\{ t\in [s_0,s_n]\Big|\forall s\in [t,s_n], \normX{Q^n(s)-R^{hy}_{\gc^*,\ga^*+\gc^*s}}\leq \alpha_0\right\}.
    \end{equation}

Thus $s_n\in J_n$ and by continuity of the flow, there exists $\tau_n$ such that $[s_n-\tau_n,s_n]\subset J_n$, thus $t^*\in J_n$. Assume by contradiction that $t^*>s_0$, then by Lemma~\ref{proposition: reduction proposition uniform estimates}, for any $t\in [t^*,s_n]$
\begin{equation}
         \normX{Q^n(t)-R^{hy}_{\gc^*,\ga^*+\gc^*t}}\leq K_0 e^{-\tau_0 t}\leq K_0 e^{-\tau_0 s_0}\leq \dfrac{\alpha_0}{2}.
     \end{equation}

Once more, by continuity of the flow, there exists $t_1^*<t^*$ such that for $t\in [t_1^*,t^*]$, we have
\begin{equation}
         \normX{Q^n(t)-R^{hy}_{\gc^*,\ga^*+\gc^*t}}\leq \dfrac{3\alpha_0}{4},
     \end{equation}
     which contradicts the minimality of $t^*$. Thus $t^*=s_0$ and $ \Vert Q^n(t)-R^{hy}_{\gc^*,\ga^*+\gc^* t}\Vert_{\mathcal{X}_{hy}}\leq K_0 e^{-\tau_0 t}$ for any $t\in [s_0,s_n]$.

\end{proof}

\begin{merci}
I am very grateful to Philippe Gravejat for his continued interest in my work and for the many pieces of advice he has given me. I would also like to thank Raphaël Côte for helping me with various technical aspects and for being such a great discussion partner.
\end{merci}

\appendix

\bibliographystyle{plain}
\bibliography{bibliography}

\end{document}